\documentclass[11pt,reqno]{amsart}

\usepackage{amsmath, amsthm, amsopn, amssymb, graphicx, enumerate, enumitem, geometry, afterpage, caption, subcaption, color, textcomp, stmaryrd, bbm, tikz, epstopdf}
\usepackage[nobysame,msc-links]{amsrefs}
\usepackage[english]{babel}
\usepackage{mathtools}
\usetikzlibrary{arrows,automata}
\usetikzlibrary{decorations,decorations.pathmorphing}

\newtheorem{thm}{Theorem}[section]
\newtheorem{lemma}[thm]{Lemma}

\newtheorem{cor}[thm]{Corollary}

\theoremstyle{definition}

\newtheorem{conj}[thm]{Conjecture}
\newtheorem{ques}[thm]{Question}

\newcommand{\R}{\mathbb{R}}
\newcommand{\Z}{\mathbb{Z}}

\renewcommand{\P}{\mathbb{P}}
\renewcommand{\S}{\mathcal{S}}

\makeatletter
\DeclareRobustCommand{\cev}[1]{%
	\mathpalette\do@cev{#1}%
}
\newcommand{\do@cev}[2]{%
	\fix@cev{#1}{+}%
	\reflectbox{$\m@th#1\vec{\reflectbox{$\fix@cev{#1}{-}\m@th#1#2\fix@cev{#1}{+}$}}$}%
	\fix@cev{#1}{-}%
}
\newcommand{\fix@cev}[2]{%
	\ifx#1\displaystyle
	\mkern#23mu
	\else
	\ifx#1\textstyle
	\mkern#23mu
	\else
	\ifx#1\scriptstyle
	\mkern#22mu
	\else
	\mkern#22mu
	\fi
	\fi
	\fi
}
\makeatother

\DeclareMathOperator\diam{diam}

\newcommand{\eps}{\varepsilon}

\DeclareMathOperator{\vol}{vol}

\makeatletter
\newcommand*\rel@kern[1]{\kern#1\dimexpr\macc@kerna}
\newcommand*\widebar[1]{%
  \begingroup
  \def\mathaccent##1##2{%
    \rel@kern{0.8}%
    \overline{\rel@kern{-0.8}\macc@nucleus\rel@kern{0.2}}%
    \rel@kern{-0.2}%
  }%
  \macc@depth\@ne
  \let\math@bgroup\@empty \let\math@egroup\macc@set@skewchar
  \mathsurround\z@ \frozen@everymath{\mathgroup\macc@group\relax}%
  \macc@set@skewchar\relax
  \let\mathaccentV\macc@nested@a
  \macc@nested@a\relax111{#1}%
  \endgroup
}
\makeatother

\title{On the site percolation threshold of circle packings and planar graphs}
\date{\today}

\author{Ron Peled}
\address{Ron Peled\hfill\break
	Tel Aviv University\\
	School of Mathematical Sciences\\
	Tel Aviv, 69978, Israel.}
\email{peledron@tauex.tau.ac.il}
\urladdr{http://www.math.tau.ac.il/~peledron}

\begin{document}

\begin{abstract}
A circle packing is a collection of disks with disjoint interiors in the plane. It naturally defines a graph by tangency. It is shown that there exists $p>0$ such that the following holds for every circle packing: If each disk is retained with probability $p$ independently, then the probability that there is a path of retained disks connecting the origin to infinity is zero. The following conclusions are derived using results on circle packings of planar graphs: (i) Site percolation with parameter $p$ has no infinite connected component on recurrent simple plane triangulations, or on Benjamini--Schramm limits of finite simple planar graphs. (ii) Site percolation with parameter $1-p$ has an infinite connected component on transient simple plane triangulations with bounded degree. These results lend support to recent conjectures of Benjamini. Extensions to graphs formed from the packing of shapes other than disks, in the plane and in higher dimensions, are presented. Several conjectures and open questions are discussed. 

\end{abstract}

\maketitle

\section{Introduction}\label{sec:introduction}
\emph{Site percolation with parameter $p\in[0,1]$ on a graph $G$} is the process of independently retaining each vertex in $G$ with probability $p$ and deleting it with probability $1-p$; the (random) induced subgraph on the retained vertices is denoted $G^p$. Percolation theory is concerned with the structure of the connected components of $G^p$. It forms a huge body of research in both the physics and mathematics communities; see \cites{G99, BR06, DC17} for two books and a recent review.
Possibly the most basic question in the theory is whether $G^p$ has a connected component with infinitely many vertices. Kolmogorov's zero-one law implies that the probability of this event is either zero or one for each $p$ (on locally finite graphs), and a standard coupling shows that the probability is non-decreasing in $p$. This leads to the definition of the \emph{critical probability $p_c(G)$}, defined as the infimum over all $p$ such that $G^p$ has an infinite connected component almost surely.

The critical probability of the triangular lattice is exactly $1/2$, but for most lattices (and more general graphs) the critical probability is not predicted to have an explicit expression; simulations suggest that $p_c(\Z^2)\approx0.59$. The critical probability of general graphs, or even planar graphs, may be arbitrarily small, as evidenced by the fact that the critical probability of the $d$-regular tree is $\frac{1}{d-1}$. The lower bound $p_c(G)\ge \frac{1}{\Delta(G)-1}$ holds for graphs $G$ with finite maximal degree $\Delta(G)$, as follows from a simple union bound or by coupling the percolation on the graph with a percolation on the regular tree. For which general classes of graphs can this lower bound be improved? Motivated by ideas of coarse conformal uniformization,
Benjamini~\cite{B2018} recently made two conjectures on the behavior of site percolation with parameter $p=1/2$ on large classes of planar triangulations (Conjecture~\ref{conj:transient triangulations} and Conjecture~\ref{conj:square tiling} below). The first conjecture relates the behavior to the recurrence properties of the triangulation (a connected graph is called \emph{recurrent} if simple random walk on it returns to its starting vertex infinitely often, almost surely) while the second discusses connectivity probabilities in specific embeddings.

The present work is motivated by Benjamini's conjectures, as well as by applications to the study of planar loop models~\cite{CGHP20}. We prove a (positive) uniform lower bound on the critical probability of locally finite planar graphs which can be represented as the tangency graph of a circle packing with at most countably many accumulation points (Corollary~\ref{cor:circle packing accumulation points}). This class of graphs includes \emph{all recurrent simple plane triangulations} as well as \emph{all Benjamini--Schramm limits of finite planar graphs}. A uniform \emph{upper bound} on the critical probability is further obtained for \emph{transient} simple plane triangulations with bounded degrees. In each of these cases, the fact that the critical probability cannot be arbitrarily close to $0$, or arbitrarily close to $1$ in the latter case, was not known before. The results lend support to Benjamini's conjectures, verifying their analogues when the percolation probability 1/2 is replaced by a different universal constant (close to $0$ or close to $1$, according to context).

Our results are based on the following statement (Theorem~\ref{thm:percolation for circle packing}): There exists $p>0$ such that for \emph{all} circle packings in the plane, after retaining each disk with probability $p$ and deleting it with probability $1-p$, there is no path of retained disks connecting the origin to infinity, almost surely. In the spirit of Benjamini's conjectures, we conjecture that one may in fact take $p=1/2$ in this statement (see Section~\ref{sec:discussion and open questions}). Similar statements are obtained for packings of general shapes, in dimension two or higher, satisfying a regularity assumption.

%
%
%
%
%
%
%

\section{Results}

\subsection{Circle packings}\label{sec:circle packings}
A \emph{circle packing} is a collection of closed (geometric) disks in $\R^2$ having positive radii (possibly changing from disk to disk) and disjoint interiors. A circle packing $\S$ naturally defines a graph $G_\S$ with vertex set $\S$ by declaring disks \emph{adjacent} when they are tangent. We also write that a graph \emph{$G$ is represented by $\S$} if $G = G_\S$. We note that a circle packing may have \emph{accumulation points} - points in $\R^2$ with infinitely many disks of the packing intersecting each of their $\R^2$-neighborhoods. In addition, a disk may be tangent to infinitely many other disks.

As before, we denote by $G_\S^p$ the (random) induced subgraph on retained disks in a site percolation process on $G_\S$ with parameter $p$. The retained disks are termed \emph{open} and the non-retained ones, \emph{closed}. Given $s_0, s_1\in\S$ write $s_0\xleftrightarrow{\S, p} s_1$ for the event that $s_0$ and $s_1$ are connected in $G_\S^p$, i.e., that there is a finite path of open disks between them (in particular, $s_0$ and $s_1$ need to be open). Define the distance between $s_0$ and $s_1$ by
\begin{equation}\label{eq:distance between sets}
  d(s_0, s_1):=\min\{\|x-y\|_\infty\colon x\in s_0, y\in s_1\}
\end{equation}
where $\|\cdot\|_\infty$ denotes $\ell_\infty$ distance; we emphasize that this distance is measured \emph{in the ambient space $\R^2$} rather than in the graph~$G_\S$. Given $r>0$ and $s_0\in \S$ define the event that $s_0$ is connected by open disks to some $s\in\S$ at distance at least $r$ from it,
\begin{equation}\label{eq:connectivity event}
  E_{\S, p}(s_0, r) := \{\exists s\in\S\text{ satisfying }d(s_0, s)\ge r\text{ and }s_0\xleftrightarrow{\S, p} s\}.
\end{equation}
We also define the event that $s_0$ is connected to infinity; precisely, let $E_{\S, p}(s_0, \infty)$ be the event that $s_0$ is open and there is a sequence of open $s_1,s_2,\ldots$ in $\S$ with $s_n$ adjacent to $s_{n+1}$ in $G_\S$ for $n\ge 0$ and with $d(s_0, s_n)\to\infty$ as $n\to\infty$ (see Section~\ref{sec:other connectivity notions} for other connectivity notions). Lastly, let $\diam(s)$ be the diameter of a disk $s$.
The following is our main result.

\begin{thm}\label{thm:percolation for circle packing}
  There exists $p>0$ such that the following holds: Let $\S$ be a circle packing and $s_0\in \S$. Then
  \begin{equation}\label{eq:no open path to infinity circle packing}
    \P(E_{\S, p}(s_0, \infty))=0.
  \end{equation}
  Moreover, if $D:=\sup_{s\in\S}\diam(s)<\infty$ then for each $r>0$,
  \begin{equation}\label{eq:bounded above squares circle packing}
    \P(E_{\S, p}(s_0, r))\le e^{-\frac{r}{D}}.
  \end{equation}
\end{thm}

We emphasize that the theorem excludes only the existence of an open path to \emph{infinity} in the percolation process on $\S$. If $\S$ has accumulation points, it may still happen that $G_{\S}^p$ has an infinite connected component. Indeed, for each $p>0$ there are circle packings for which this is the case, with probability one --- e.g., circle packings of regular trees of sufficiently high degree. Another example is when $\S$ has a disk $s$ which is tangent to infinitely many other disks, in which case $G_{\S}^p$, for any $p>0$, has an infinite connected component with positive probability (with probability one on the event that $s$ is open). Infinite connected components of $G_\S^p$ are excluded, however, when $\S$ has at most countably many accumulation points and $\S$ is \emph{locally finite} in the sense that each disk of $\S$ is tangent to only finitely many other disks.
\begin{cor}\label{cor:circle packing accumulation points}
  Let $p>0$ be the constant from Theorem~\ref{thm:percolation for circle packing}. Let $\S$ be a locally finite circle packing with at most countably many accumulation points in $\R^2$. Then the probability that $G_\S^p$ has an infinite connected component is zero.
\end{cor}
The corollary follows from Theorem~\ref{thm:percolation for circle packing} by applying M\"obius transformations to $\S$, as detailed in Section~\ref{sec:proof of corollary no accumulation points}.


\subsection{Recurrent and transient triangulations}\label{sec:recurrent and transient triangulations}
In this section we deduce consequences of Theorem~\ref{thm:percolation for circle packing} for plane triangulations. A main tool is a result of He and Schramm~\cite{HS95} which connects the recurrence/transience properties of plane triangulations with their representing circle packings. The obtained results further shed light on two conjectures and a question of Benjamini~\cite{B2018}.


We start by describing required notation, following~\cite[Chapters 3,4]{N18}.
A graph is \emph{locally finite} if every vertex has finite degree, it is \emph{simple} if it has no multiple edges or self loops and it is of \emph{bounded degree} if the supremum of its degrees is finite. A \emph{proper drawing} of a planar graph $G$ is a map sending the vertices of $G$ to distinct points in $\R^2$ and the edges of $G$ to continuous curves between the corresponding vertices so that no two curves intersect, except at the vertices shared by their edges. A \emph{planar map} is a locally finite planar graph endowed with a cyclic permutation of the edges incident to each vertex, such that there exists a proper drawing of the graph for which the clockwise order of the curves touching the image of each vertex follows the cyclic permutation associated to that vertex. We often use the same notation for the planar map and its underlying graph. A planar map is \emph{simple} (\emph{connected, bounded degree}) if its underlying graph is simple (connected, bounded degree). The structure of a planar map allows to define its \emph{faces}. To this end direct each of the edges of $G$ in both ways and say that a directed edge $\vec{e}$ of the map \emph{precedes} the directed edge $\vec{f}$ if $\vec{e}=(x,v), \vec{f}=(v,y)$ and $y$ is the successor of $x$ in the cyclic permutation of $v$ (if $x$ is the only neighbor of $v$ we mean that $y=x$). Now define an equivalence relation on directed edges by saying that $\vec{e}$ and $\vec{f}$ are in the same face if there exists a directed path $\vec{e}_1,\ldots, \vec{e}_m$ in the graph with $\{\vec{e},\vec{f}\}=\{\vec{e}_1,\vec{e}_m\}$ and $\vec{e}_i$ preceding $\vec{e}_{i+1}$ for each $i$. Faces are defined as the equivalence classes of this relation. The planar map is a \emph{triangulation} if it is connected and each of its faces has exactly 3 edges. A graph (or planar map) is \emph{infinite} it has infinitely many vertices. An infinite graph is \emph{one-ended} if removing any finite subset of its vertices (and their incident edges) leaves exactly one infinite connected component. Infinite one-ended triangulations are also called \emph{plane triangulations} or \emph{disk triangulations} as they have proper drawings which `cover' the plane, or disk, in a suitable sense (see~\cite[Chapter 4.1]{N18}).

Itai Benjamini~\cite{B2018} (see also~\cite{BK18}) made several conjectures regarding percolation on planar triangulations which are suggested by assuming quasi-invariance under coarse conformal uniformization. The following conjecture and question serve as part of the motivation for our next results.
\begin{conj}\cite[Conjecture 2.1]{B2018}\label{conj:transient triangulations}
  Let $G$ be a bounded degree plane triangulation. If $G$ is transient then $G^{1/2}$ has an infinite connected component almost surely.
\end{conj}
It is further pointed out in~\cite{B2018} that the conjecture remains open even if $1/2$ is replaced by any other fixed $1/2<p<1$.
\begin{ques}\cite[Section 2]{B2018}\label{ques:recurrent triangulations}
  Let $G$ be a bounded degree plane triangulation. Does recurrence of $G$ imply that $G^{1/2}$ does not have an infinite connected component almost surely?
\end{ques}

We say that a planar map $G$ is \emph{represented} by a circle packing $\mathcal{S}$ if $G = G_{\S}$ and the cyclic order on the edges incident to each vertex $v$ equals the clockwise order in which the disks corresponding to the neighbors of $v$ appear around the disk corresponding to $v$. The \emph{carrier} of a circle packing $\S$ representing a triangulation is the union of the closed disks of $\S$ together with the space between any three disks corresponding to a face of the triangulation. The definitions imply that all accumulation points of $\S$ lie outside of its carrier.

The following theorem is part of the main result of~\cite{HS95} (see also~\cite[Chapter 4]{N18}). Write $\mathbb{D}:=\{z\in\R^2\colon \|z\|<1\}$ for the open unit disk.
\begin{thm}(He--Schramm~\cite{HS95})\label{thm:He--Schramm}
Let $G$ be a simple plane triangulation.
\begin{enumerate}
  \item If $G$ is recurrent then it may be represented by a circle packing whose carrier is $\R^2$.
  \item If $G$ is transient and of bounded degree then it may be represented by a circle packing whose carrier is $\mathbb{D}$.\label{item:He--Schramm transience}
\end{enumerate}
\end{thm}
The first part of Theorem~\ref{thm:percolation for circle packing}, given in~\eqref{eq:no open path to infinity circle packing}, is not helpful for circle packings whose carrier is the unit disk. However, the second part, given in~\eqref{eq:bounded above squares circle packing}, is applicable and allows to derive the following lemma, which shows the existence of an infinite cluster for sufficiently large values of the percolation parameter.
\begin{lemma}\label{lem:circle packings of unit disk}
  Let $p>0$ be the constant from Theorem~\ref{thm:percolation for circle packing}. Let $\S$ be a circle packing representing a triangulation whose carrier is $\mathbb{D}$. Then the probability that $G_\S^{1-p}$ has an infinite connected component is one.
\end{lemma}
The method of proof of the lemma can yield the existence of infinite connected components on other triangulations. We formulate one more result of this type.
\begin{lemma}\label{lem:circle packings with slowly growing radii}
  Let $p>0$ be the constant from Theorem~\ref{thm:percolation for circle packing}. There exists $c>0$ such that the following holds. Let $\S$ be a circle packing representing a triangulation whose carrier is $\mathbb{\R}^2$. Let $f(r)$ be the maximal radius of a disk in $\S$ intersecting $r\mathbb{D}$. Suppose that
  \begin{equation}\label{eq:radius growth}
     \limsup_{r\to\infty} f(r)\cdot\frac{\log\log r}{r}\le c.
   \end{equation}
   Then the probability that $G_\S^{1-p}$ has an infinite connected component is one.
\end{lemma}
We remark that the assertion of the lemma may fail if $f(r)$ grows linearly; see Section~\ref{sec:value of p_c}.

Corollary~\ref{cor:circle packing accumulation points}, Theorem~\ref{thm:He--Schramm} and Lemma~\ref{lem:circle packings of unit disk} yield the following corollary, which resolves the versions of Conjecture~\ref{conj:transient triangulations} and Question~\ref{ques:recurrent triangulations} in which the parameter $1/2$ is replaced by sufficiently large and sufficiently small probabilities, respectively.
\begin{cor}\label{cor:percolation on plane triangulations}
  Let $p>0$ be the constant from Theorem~\ref{thm:percolation for circle packing}. Let $G$ be a simple plane triangulation.
  \begin{enumerate}
    \item If $G$ is recurrent then the probability that $G^p$ has an infinite connected component is zero.\label{item:recurrent plane triangulations}
    \item If $G$ is transient and has bounded degrees then the probability that $G^{1-p}$ has an infinite connected component is one.\label{item:transient plane triangulations}
  \end{enumerate}
\end{cor}

Benjamini made a second conjecture regarding left-right crossings in percolation on square tilings.
\begin{conj}\cite[Conjecture 2.2]{B2018}\label{conj:square tiling} There exists $c>0$ so that the following holds. Tile the unit square with (possibly infinitely many) squares of varying sizes so that at most three squares meet at corners. Color each square black or white with equal probability independently. Then the probability of a black left-right crossing is at least $c$.
\end{conj}
It is pointed out in~\cite{B2018} that the conjecture is open even when the probability to color a square black is $2/3$, and to the author's knowledge the conjecture is open for all fixed probabilities in $[1/2,1)$. Our results suffice to verify the version of the conjecture in which the probability of coloring a square black is an absolute constant close to $1$ and the tiling involves only finitely many squares (otherwise, the notion of left-right crossing may need to be made more precise; see Section~\ref{sec:other connectivity notions}). We prove a slightly stronger statement. Note that as the squares in Conjecture~\ref{conj:square tiling} form a triangulation, the event of a black left-right crossing equals the event that there is no white top-bottom crossing. We consider a collection of squares with disjoint interiors in the unit square (not necessarily a tiling, with four squares allowed to share a corner) and bound the probability of a white top-bottom crossing where the crossing is allowed to use diagonal connectivity when squares share a corner diagonally. The statement (and proof) involve Theorem~\ref{thm:percolation for general packings} of Section~\ref{sec:general packings}.
\begin{cor}\label{cor:square tiling of square}
  There exists $c>0$ so that the following holds. Let $p>0$ be the constant from Theorem~\ref{thm:percolation for general packings} for a packing of squares in $\R^2$. Pack finitely many squares in the unit square. Color each square white with probability $p$ independently. Then the probability of a white top-bottom crossing of the unit square (with diagonal connectivity allowed when four squares share a corner) is at most $1-c$.
\end{cor}
Benjamini~\cite{B2018} remarks that if conjecture~\ref{conj:square tiling} is true then the same should hold for a tiling, or a packing of a triangulation, with a set of shapes that are of bounded Hausdorff distance to circles. We note that a version of Corollary~\ref{cor:square tiling of square} with other shapes may be proved in a similar manner using Theorem~\ref{thm:percolation for general packings}.


\subsection{Benjamini--Schramm limits}
In this section we use Corollary~\ref{cor:circle packing accumulation points} to study percolation on Benjamini--Schramm limits of finite planar graphs. We conclude that on \emph{all} such graphs, percolation with the parameter $p$ of Theorem~\ref{thm:percolation for circle packing} has no infinite connected component, almost surely.


We start by defining the necessary concepts (see also~\cite[Chapter 5]{N18}).
A \emph{rooted graph} $(G,\rho)$ is a graph with a distinguished vertex $\rho$.
For a graph $G$, vertex $v$ of $G$ and integer $r\ge 0$, let $B_G(v,r)$ be the graph ball of radius $r$ in $G$ centered around $v$, that is, the induced subgraph on the set of vertices at graph distance at most $r$ from $v$, rooted at $v$. Suppose $(G_n)_{n\ge 1}$ is a sequence of, possibly random, finite graphs. Let $\rho_n$ be a uniformly sampled vertex of $G_n$ (if $G_n$ is random, one first samples $G_n$ and then samples $\rho_n$ uniformly in $G_n$). Let $(G,\rho)$ be a random rooted graph, with $G$ connected almost surely. Then $(G,\rho)$ is called the \emph{Benjamini--Schramm limit} (or local limit) of $(G_n)$ if for each $r\ge 0$, the distribution of $B_{G_n}(\rho_n, r)$ converges as $n$ tends to infinity to the distribution of $B_G(\rho, r)$ (in the sense that for each $r\ge 0$, $\P(B_{G_n}(\rho_n,r)=(H,\sigma))\to\P(B_G(\rho,r)=(H,\sigma))$ for every rooted graph $(H,\sigma)$, where the equality sign denotes the existence of a root-preserving graph isomorphism). Note that if $(G,\rho)$ is the Benjamini--Schramm limit of the finite graphs $(G_n)$ then $G$ is locally finite, and that if the $(G_n)$ are simple then also $G$ is simple.


Benjamini and Schramm~\cite{BS01} proved that every Benjamini--Schramm limit of finite simple \emph{planar} graphs with uniformly bounded degrees may be represented by a circle packing with at most one accumulation point, almost surely. Asaf Nachmias explained to the author that the restriction on the degrees is not required for this conclusion (Benjamini and Schramm further proved that the limiting graph is recurrent, and for this conclusion the restriction is relevant) and suggested the following lemma and the idea of its proof.
\begin{lemma}\label{lem:circle packing with at most one accumulation point}
  Let $(G_n)$ be a sequence of, possibly random, finite simple planar graphs with Benjamini--Schramm limit $(G,\rho)$. Then, almost surely, there is a circle packing $\S$ with at most one accumulation point in $\R^2$ such that $G=G_\S$.
\end{lemma}
Combining Corollary~\ref{cor:circle packing accumulation points} and Lemma~\ref{lem:circle packing with at most one accumulation point} yields the following conclusion.
\begin{cor}\label{cor:percolation on Benjamini Schramm limits}
  Let $p>0$ be the constant from Theorem~\ref{thm:percolation for circle packing}. Let $(G,\rho)$ be a Benjamini--Schramm limit of, possibly random, finite simple planar graphs. Then the probability that $G^p$ has an infinite connected component is zero.
\end{cor}
This corollary is an important ingredient for the proof in~\cite{CGHP20} that the loop $O(n)$ model exhibits macroscopic loops in a subset of positive measure of its phase diagram. This application makes use of the fact that the corollary holds without a bounded degree assumption.

\subsection{General packings}\label{sec:general packings}
The method of proof of Theorem~\ref{thm:percolation for circle packing} generalizes to higher dimensions and to a general class of packings. We describe a theorem to this effect following required notation.

A \emph{packing} in $\R^d$ is a finite or countable collection of non-empty compact sets in $\R^d$ with disjoint interiors. A packing $\S$ defines a graph $G_\S$ with vertex set $\S$ by declaring two (distinct) sets \emph{adjacent} when their intersection is non-empty. It is noted that these general definitions allow for sets in $\S$ to be disconnected, that even when $d=2$ the graph $G_\S$ need not be planar and that a set may be adjacent to infinitely many other sets in $G_\S$.

Again, denote by $G_\S^p$ the (random) induced subgraph on retained sets in a site percolation process on $G_\S$ with parameter $p$. Given $s_0, s_1\in\S$ write $s_0\xleftrightarrow{\S, p} s_1$ for the event that $s_0$ and $s_1$ are connected in $G_\S^p$, i.e., that there is a finite path of open sets between them (in particular, $s_0$ and $s_1$ need to be open). The distance $d(s_0, s_1)$ between sets is again given by~\eqref{eq:distance between sets} where now $\|\cdot\|_\infty$ denotes the $\ell_\infty$ distance in $\R^d$. The event $E_{\S,p}(s_0, r)$ that $s_0$ is connected by open sets to some $s\in\S$ at distance at least $r$ from it is again defined by~\eqref{eq:connectivity event}. The event that $s_0$ is connected to infinity, $E_{\S, p}(s_0, \infty)$, is again the event that $s_0$ is open and there is a sequence of open $s_1,s_2,\ldots$ in $\S$ with $s_n$ adjacent to $s_{n+1}$ in $G_\S$ for $n\ge 0$ and with $d(s_0, s_n)\to\infty$ as $n\to\infty$. For a set $s$ in $\R^d$ define its diameter by
\begin{equation*}
  \diam(s):=\sup\{\|x-y\|_\infty\colon x,y\in s\}.
\end{equation*}

For a measurable set $s$ in $\R^d$ let $\vol(s)$ be its Lebesgue measure and $\partial s$ stand for its boundary. A packing $\S$ in $\R^d$ is called \emph{$\eps$-regular} if $\vol(\partial s)=0$ and $\vol(s)\ge \eps \diam(s)^d$ for each $s\in\S$.
\begin{thm}\label{thm:percolation for general packings}
  Let $d\ge 2$ and $\eps>0$. There exists $p = p(d,\eps)>0$ such that the following holds: Let $\S$ be an $\eps$-regular packing in $\R^d$ and $s_0\in \S$. Then
  \begin{equation}\label{eq:no open path to infinity general packing}
    \P(E_{\S, p}(s_0, \infty))=0.
  \end{equation}
  Moreover, if $D:=\sup_{s\in\S}\diam(s)<\infty$ then for each $r>0$,
  \begin{equation}\label{eq:bounded above diameters general packing}
    \P(E_{\S, p}(s_0, r))\le e^{-\frac{r}{D}}.
  \end{equation}
\end{thm}
We remark that the exponential decay bound~\eqref{eq:bounded above diameters general packing} has no leading constant in the exponent. However, examination of the proof shows that the bound $\P(E_{\S, p}(s_0, r))\le e^{-C\frac{r}{D}}$ with $C\ge 1$ is also valid upon replacing $p(d,\eps)$ with a suitable $p(d,\eps,C)>0$.

Our techniques can lead to a more general result, in which the sets are arbitrary subsets of some metric space (without the packing assumption) and restrictions are placed on the number of sets having diameter at least $d_1$ which can be at distance at most $k\cdot d_1$ from a set of diameter $d_2$. We do not elaborate on this extension.

\subsection{Acknowledgements} The author is grateful to Asaf Nachmias for enlightening discussions on the topic of circle packings and for suggesting the statement and idea of proof of Lemma~\ref{lem:circle packing with at most one accumulation point}. Further thanks are due to Itai Benjamini for enthusiasm and support and to Alexander Glazman, Ori Gurel-Gurevich and Matan Harel for very helpful conversations on the presented results.

The author is supported by the Israel Science Foundation grants~861/15 and~1971/19 and by the European Research Council starting grant 678520 (LocalOrder).

\section{Percolation on square packings}\label{sec:percolation on square packings}

A \emph{square packing} $\S$ is a collection of closed squares in $\R^2$ with sides parallel to the coordinate axes and with disjoint interiors. It is a special case of the $\eps$-regular packings discussed in Section~\ref{sec:general packings}.
To explain the argument in its simplest form first, we start by explaining the proof of Theorem~\ref{thm:percolation for general packings} for square packings. The full theorem, of which the case of circle packings is a special case, is proved in Section~\ref{sec:general packings proof}.

\begin{proof}[Proof of Theorem~\ref{thm:percolation for general packings} for square packings]
  We prove the theorem with
  \begin{equation}\label{eq:choice of p}
    p := e^{-26}
  \end{equation}
  (though fine-tuning the proof can give somewhat larger values).

  As the first and main step of the proof, we show (a variant of) the bound~\eqref{eq:bounded above diameters general packing} under the additional assumption that the side lengths of the squares in the packing are bounded below: For integer $k\ge 0$ let $\Sigma^k$ be the collection of all pairs $(\S,s_0)$ with $\S$ a square packing with side lengths in $[1,2^k]$ and $s_0\in\S$. Define
  \begin{equation}\label{eq:alpha def}
    \alpha(k,r):=\sup_{(\S, s_0)\in\Sigma^k}\P(E_{\S, p}(s_0, r)).
  \end{equation}
  We shall prove that for integer $k\ge 0$ and real $r>0$,
  \begin{equation}\label{eq:induction statement}
    \alpha(k,r)\le e^{-\frac{r}{2^{k-1}}}.
  \end{equation}

  We prove~\eqref{eq:induction statement} by a double induction on $k$ and $r$. We start with the base case $k=0$. Let $r>0$ and let $(\S, s_0)\in \Sigma^0$. As all squares in $\S$ have side length $1$, if $s\in\S$ satisfies $d(s_0, s)\ge r$ then the graph distance of $s_0$ and $s$ in $G_\S$ is at least $\lceil r\rceil+1$. Simple geometric considerations show that each square in $\S$ can be adjacent to at most $8$ other squares in $G_\S$ (recalling that squares touching at a corner are adjacent) and thus the number of paths of length $L$ in $G_\S$ which start at $s_0$ is at most $8^L$. We conclude that
  \begin{equation*}
    \P(E_{\S, p}(s_0,r))\le \P(\text{there is a simple path in $G_\S^p$ of length $\lceil r\rceil+1$ from $s_0$})\le p\cdot(8p)^{\lceil r\rceil+1}\le e^{-2r}
  \end{equation*}
  for our choice~\eqref{eq:choice of p} of $p$. Thus~\eqref{eq:induction statement} is established for $k=0$ and all $r>0$.

  Fix an integer $k\ge 1$. We assume by induction that~\eqref{eq:induction statement} is established for all $r>0$ when the $k$ of~\eqref{eq:induction statement} is replaced by $k-1$, and proceed to establish~\eqref{eq:induction statement} for all $r>0$ with our fixed $k$. This is achieved via a second induction on $r$. Suppose first, as a base case, that $0<r\le 2^k$. Then for each $(\S, s_0)\in \Sigma^k$, recalling the choice~\eqref{eq:choice of p} of $p$,
  \begin{equation}\label{eq:small r}
    \P(E_{\S, p}(s_0,r))\le \P(s_0\in G_\S^p) = p <e^{-2}\le e^{-\frac{r}{2^{k-1}}}
  \end{equation}
  as required. Now fix $r>2^k$ and assume, as the second induction hypothesis, that~\eqref{eq:induction statement} is established with our fixed $k$ when the $r$ of~\eqref{eq:induction statement} is restricted to $[0,r - 2^k]$.

  Let $(\S, s_0)\in \Sigma^k$. We first reduce to the case in which $\diam(s_0)\in[1, 2^{k-1}]$, if this is not already the case, by the following geometric construction (this is the place where it is convenient to work with squares instead of disks). If $\diam(s_0)\in(2^{k-1}, 2^k]$ we cut $s_0$ in half along both axes, dividing it into four squares $s_0^1, \ldots, s_0^4$ with half the side length of $s_0$ and with disjoint interiors. These give rise to pairs $(\S^i, s_0^i)$, $1\le i\le 4$, where $\S^i := (\S\setminus\{s_0\})\cup\{s_0^i\}$. We use the natural coupling of $G_{\S}^p$ with $G_{\S^i}^p$ in which the vertex sets of these (random) graphs are the same except that $s_0\in V(G_{\S}^p)$ if and only if $s_0^i\in V(G_{\S^i}^p)$. The construction yields that $E_{\S,p}(s_0,r)\subseteq\cup_{i=1}^4 E_{\S^i,p}(s_0^i, r)$ whence
  \begin{equation*}
    \P(E_{\S,p}(s_0,r))\le \sum_{i=1}^4 \P(E_{\S^i,p}(s_0^i, r)).
  \end{equation*}
  It thus suffices to prove that for each $(\S, s_0)\in \Sigma^k$ with $\diam(s_0)\in[1,2^{k-1}]$ we have
  \begin{equation}\label{eq:goal after dividing s 0}
    \P(E_{\S,p}(s_0,r))\le \frac{1}{4}e^{-\frac{r}{2^{k-1}}}.
  \end{equation}
  Fix such an $(\S, s_0)$. We next aim to decompose the event $E_{\S, p}(s_0, r)$ into events to which our induction hypotheses apply. To this end, decompose $\S$ into
  \begin{equation*}
  \begin{split}
    \S^{k-1} &:= \{s\in \S\colon \diam(s)\in[1,2^{k-1}]\},\\
    \S_{k-1}^k &:= \{s\in \S\colon \diam(s)\in(2^{k-1},2^k]\}.
  \end{split}
  \end{equation*}
  Note that $s_0\in\S^{k-1}$ by assumption. We couple $G_{\S^{k-1}}^p$ and $G_{\S}^p$ in the natural way, by setting $G_{\S^{k-1}}^p$ to be the induced subgraph of $G_\S^p$ on $\S^{k-1}$. Let
  \begin{equation*}
    E^{k-1} := E_{\S^{k-1},p}(s_0, r)
  \end{equation*}
  be the event that $s_0$ is connected to $\ell^\infty$ distance $r$ by a path in $G_{\S^{k-1}}^p$. In addition, define three events for each $s\in \S_{k-1}^k$: Let $E_s^1$ be the event that there is a path in $G_{\S^{k-1}}^p$ from $s_0$ to a square adjacent (in $G_\S$) to $s$. Let $E_s^2$ be the event that there is a path in $G_{\S}^p$ from $s$ to some $s_1$ with $d(s_0, s_1)\ge r$ (allowing the possibility that $s_1=s$), and let $E_s$ be the event that $E_s^1$ and $E_s^2$ occur disjointly, that is, that there exist two disjoint paths of open squares, one implying that $E_s^1$ occurs and the other implying that $E_s^2$ occurs (the notation $E_s := E_s^1\circ E_s^2$ is sometimes used for this operation). Our definitions imply the following decomposition
  \begin{equation*}
    E_{\S,p}(s_0, r) = E^{k-1}\cup\bigcup_{s\in \S_{k-1}^k} E_s
  \end{equation*}
  so that
  \begin{equation}\label{eq:E S p union bound}
    \P(E_{\S,p}(s_0, r))\le \P(E^{k-1}) + \sum_{s\in \S_{k-1}^k} \P(E_s).
  \end{equation}
  We proceed to estimate each of the probabilities on the right-hand side. First,
  \begin{equation}\label{eq:E k-1 probability bound}
    \P(E^{k-1}) \le \alpha(k-1,r)\le e^{-\frac{r}{2^{k-2}}}
  \end{equation}
  by the (first) induction hypothesis. Second, as the two paths involved in the definition of $E_s$ are disjoint, we may invoke the van-den-Berg--Kesten inequality~\cite{BK85} to obtain
  \begin{equation*}
    \P(E_s) \le \P(E_s^1)\cdot\P(E_s^2)\le \min\{p, \alpha(k-1, d(s_0, s)-2^{k-1})\}\cdot\alpha(k, r - d(s_0, s) - 2^k).
  \end{equation*}
  where we define $\alpha(k,r')$ to be $1$ if $r'\le 0$ and where we have used that $\P(E_s^1)\le p$ as the square $s_0$ itself needs to be open for $E_s^1$ to occur. Our induction hypotheses thus show that
  \begin{equation}\label{eq:E s bound}
    \P(E_s)\le \min\{p, e^{\frac{-d(s_0, s)+2^{k-1}}{2^{k-2}}}\}\cdot e^{\frac{-r +d(s_0, s) + 2^k}{2^{k-1}}}.
  \end{equation}
  We further decompose $\S_{k-1}^k$ to
  \begin{equation*}
  \begin{split}
    \S_{k-1}^k(0) &:= \{s\in \S_{k-1}^k\colon d(s_0, s)\le 8\cdot2^k\},\\
    \S_{k-1}^k(m) &:= \{s\in \S_{k-1}^k\colon m 2^k<d(s_0, s)\le (m+1)2^k\}, \quad m\ge 8
  \end{split}
  \end{equation*}
  and simplify the expression~\eqref{eq:E s bound} in each case,
  \begin{equation}\label{eq:E s probability estimates}
  \begin{split}
    \P(E_s)&\le p\cdot e^{-\frac{r}{2^{k-1}}+18},\ \ \quad\qquad\qquad\qquad s\in \S_{k-1}^k(0),\\
    \P(E_s)&\le e^{-\frac{r + d(s_0, s)}{2^{k-1}} + 4}\le e^{-\frac{r}{2^{k-1}} - 2m + 4},\quad s\in \S_{k-1}^k(m),\, m\ge 8.
  \end{split}
  \end{equation}
  We proceed to upper bound the size of $\S_{k-1}^k(m)$. Each $s\in\S_{k-1}^k$ has area at least $2^{2k-2}$ and is fully contained in the $\ell^\infty$ annulus $A_s$ of in-radius $d(s_0, s)$ and out-radius $2^{k-2}+d(s_0, s)+2^k$ around the center of $s_0$. As
  \begin{equation*}
    \text{area}(A_s) = (2^{k-2}+d(s_0, s)+2^k)^2 - d(s_0, s)^2 \le 2^{2k+1}+2^{k+2}d(s_0, s)
  \end{equation*}
  we conclude that
  \begin{equation}\label{eq:area estimates}
    \begin{split}
      |\S_{k-1}^k(0)|&\le \frac{2^{2k+1}+8\cdot 2^{2k+2}}{2^{2k-2}} = 136,\\
      |\S_{k-1}^k(m)|&\le \frac{2^{2k+1}+(m+1)2^{2k+2}}{2^{2k-2}} = 16m + 24\le 20m,\quad m\ge 8.
    \end{split}
  \end{equation}
  Finally, plugging the bounds~\eqref{eq:E k-1 probability bound}, \eqref{eq:E s probability estimates} and \eqref{eq:area estimates} into~\eqref{eq:E S p union bound} implies that
  \begin{equation}\label{eq:final bound}
  \begin{split}
    \P(E_{\S,p}(s_0, r)) &\le e^{-\frac{r}{2^{k-2}}} + 136 p\cdot e^{-\frac{r}{2^{k-1}}+18} + \sum_{m=8}^\infty 20m e^{-\frac{r}{2^{k-1}} - 2m + 4}\\
    &= \frac{1}{4}e^{-\frac{r}{2^{k-1}}}(4 e^{-\frac{r}{2^{k-1}}} + 544 p e^{18} + 80 e^4\sum_{m=8}^\infty m e^{-2m}).
  \end{split}
  \end{equation}
  Recalling that $r>2^k$ (the complimentary case having been discussed in~\eqref{eq:small r}) one checks that the expression inside the parenthesis in~\eqref{eq:final bound} is at most $1$ for our choice~\eqref{eq:choice of p} of $p$, thus verifying~\eqref{eq:goal after dividing s 0} and finishing the proof of~\eqref{eq:induction statement}.

  As the second step of the proof, we verify the bound~\eqref{eq:bounded above diameters general packing}.

  Suppose $\S$ is a square packing with side lengths in $[1,D]$ for some $D<\infty$ and let $s_0\in \S$. Then~\eqref{eq:induction statement}~implies \eqref{eq:bounded above diameters general packing} as for each $r>0$,
  \begin{equation}\label{eq:bounded above squares with general R}
    \P(E_{\S, p}(s_0, r))\le e^{-\frac{r}{2^{k-1}}} \le e^{-\frac{r}{D}}.
  \end{equation}
  where $k$ is an integer such that $2^{k-1}<D\le 2^k$. Moreover, if~\eqref{eq:bounded above diameters general packing} holds for a square packing $\S$ and square $s_0\in \S$ then, for any $\rho>0$, it holds also for the dilated square packing $\rho\S$ (having upper bound $\rho D$ on its side lengths) and square $\rho s_0$. Thus~\eqref{eq:bounded above diameters general packing} follows whenever the side lengths of $\S$ are bounded above and below by arbitrary finite positive numbers. Finally suppose $\S$ is a square packing with side lengths in $(0,D]$ (possibly with $\inf _{s\in \S} \diam(s) = 0$) and $s_0\in \S$. By definition, paths between squares in $\S$ are finite and thus in any such path there is a \emph{positive} minimal side length for the squares involved. It follows that $E_{\S, p}(s_0, r) = \cup_{n=n_0}^\infty E_{\S_n, p}(s_0, r)$ with $\S_n = \{s\in \S\colon \diam(s)\ge \frac{1}{n}\}$ and $n_0=\lceil 1/\diam(s_0)\rceil$. As the union is increasing we have that $\P(E_{\S, p}(s_0, r)) = \lim_{n\to\infty} \P(E_{\S_n, p}(s_0, r))$, implying~\eqref{eq:bounded above diameters general packing} in all cases.

  As the final step of the proof we proceed to show that~\eqref{eq:no open path to infinity general packing} holds.
  Define
  \begin{equation*}
    \alpha := \sup_{(\S, s_0)}\P(E_{\S, p}(s_0, \infty))
  \end{equation*}
  where the supremum is over all square packings $\S$ and $s_0\in \S$.

  Let $\S$ be a square packing and $s_0 \in \S$.
  By definition, the event $E_{\S, p}(s_0, \infty)$ entails the existence of a sequence of open squares $s_0, s_1,\ldots$ in $\S$ with $s_n$ adjacent to $s_{n+1}$ in $G_\S$ and $d(s_0, s_n)\to\infty$ as $n\to\infty$. Set $\ell_0 := \diam(s_0)$. Let $E^0_{\S, p}(s_0, \infty)$ be the sub-event of $E_{\S, p}(s_0, \infty)$ in which there exist $(s_n)$ as above with $\sup_n \diam(s_n)\le\ell_0$ and let $E^1_{\S, p}(s_0, \infty)$ be the complimentary sub-event. Write $\S^0:=\{s\in\S\colon \diam(s)\le \ell_0\}$. Clearly $E^0_{\S, p}(s_0, \infty) = E_{\S^0, p}(s_0, \infty)$ (with the natural coupling of the percolation processes) whence the bound~\eqref{eq:bounded above diameters general packing} shows that
  \begin{equation}\label{eq:E 0 prob bound}
    \P(E^0_{\S, p}(s_0, \infty))=0.
  \end{equation}

  It remains to show that also $E^1_{\S, p}(s_0, \infty)$ is of zero probability. Let $\S^{\text{big}}:=\{s\in\S\colon \diam(s)>\ell_0\}$. For $s\in \S^{\text{big}}$ let $F_s^1$ be the event that there is a path of open squares of side lengths at most $\ell_0$ from $s_0$ to a neighbor of $s$. Let $F_s := F_s^1\circ E_{\S, p}(s, \infty)$ be the event that $F_s^1$ occurs disjointly from $E_{\S, p}(s, \infty)$, that is, that the two events occur due to disjoint open paths. By definition,
  \begin{equation*}
    E^1_{\S, p}(s_0, \infty) = \cup_{s\in \S^{\text{big}}} F_s.
  \end{equation*}
  A second use of the van-den-Berg--Kesten inequality~\cite{BK85} implies that
  \begin{equation}\label{eq:E 1 prob bound}
    \P(E^1_{\S, p}(s_0, \infty))\le \sum_{s\in \S^{\text{big}}} \P(F_s^1)\P(E_{\S, p}(s, \infty))\le \alpha\sum_{s\in \S^{\text{big}}} \P(F_s^1).
  \end{equation}
  We proceed to bound the sum occurring in the last expression, arguing similarly to before. Write
  \begin{equation*}
  \begin{split}
    \S(0) &:= \{s\in \S^{\text{big}}\colon d(s_0, s)\le 8\ell_0\},\\
    \S(m) &:= \{s\in \S^{\text{big}}\colon m \ell_0<d(s_0, s)\le (m+1)\ell_0\}, \quad m\ge 8.
  \end{split}
  \end{equation*}
  Each square $s\in\S(0)$ has area at least $\ell_0^2$ in the $\ell^\infty$ ball of radius $9\frac{1}{2}\ell_0$ around the center of $s_0$. Similarly, each square $s\in\S(m)$ has area at least $\ell_0^2$ in the annulus of in-radius $m\ell_0$ and out-radius $(m+\frac{5}{2})\ell_0$ around the center of $s_0$. Thus $|\S(0)|\le 19^2$ and $|\S(m)|\le (2m+5)^2-(2m)^2 \le 24m$. Applying the bound~\eqref{eq:bounded above diameters general packing} we conclude that
  \begin{equation*}
    \sum_{s\in \S^{\text{big}}} \P(F_s^1) \le |\S(0)|p + \sum_{m=8}^\infty 24m e^{-m} < 1-\delta
  \end{equation*}
  for some $\delta>0$, due to our choice~\eqref{eq:choice of p} of $p$. Combining this result with~\eqref{eq:E 0 prob bound} and~\eqref{eq:E 1 prob bound} shows that
  \begin{equation*}
    \P(E_{\S, p}(s_0, \infty))\le (1-\delta)\alpha.
  \end{equation*}
  However, taking supremum over the square packing $\S$ and $s_0\in \S$ implies that
  \begin{equation*}
    \alpha \le (1-\delta)\alpha
  \end{equation*}
  which is only possible if $\alpha=0$, as required.
\end{proof}

\section{Percolation on general packings}\label{sec:general packings proof}

In this section we prove Theorem~\ref{thm:percolation for general packings}, detailing the necessary changes from the proof for square packings in Section~\ref{sec:percolation on square packings}.

\begin{proof}[Proof of Theorem~\ref{thm:percolation for general packings}]
  Fix $d\ge 2$ and $\eps>0$. The value $p = p(d,\eps)$ with which the theorem is proved is chosen small enough for the following arguments.

  The main step is again to prove a variant of the bound~\eqref{eq:bounded above diameters general packing} under the additional assumption that the diameters of the shapes are bounded below. For integer $k\ge 0$ let $\Sigma^k$ be the collection of all pairs $(\S,s_0)$ with $\S$ a packing in $\R^d$ satisfying $\diam(s)\in[1,2^k]$ for $s\in\S$, $s_0\in\S$ and where we require $\S\setminus\{s_0\}$ to be $\eps$-regular (i.e., $\eps$-regularity is not required of $s_0$). As before, set
  \begin{equation}\label{eq:alpha def general packing}
    \alpha(k,r):=\sup_{(\S, s_0)\in\Sigma^k}\P(E_{\S, p}(s_0, r)).
  \end{equation}
  and we shall prove that for integer $k\ge 0$ and real $r>0$,
  \begin{equation}\label{eq:induction statement general packing}
    \alpha(k,r)\le e^{-\frac{r}{2^{k-1}}}.
  \end{equation}
  We again use double induction on $k$ and $r$ to prove~\eqref{eq:induction statement general packing}. For the case $k=0$ observe that as $\diam(s)=1$ for $s\in\S$ and $\vol(\partial s)=0$, $\vol(s)\ge \eps$ for $s\in\S\setminus\{s_0\}$ it follows that each $s\in\S$ can be adjacent in $G_\S$ to at most $3^d/\eps$ sets in $\S$. The case $k=0$ thus follows in the proof for square packings by taking $p<\eps 3^{-d}$.

  Fix an integer $k\ge 1$. It is again assumed by induction that~\eqref{eq:induction statement general packing} is established for all $r>0$ when the $k$ of~\eqref{eq:induction statement general packing} is replaced by $k-1$. The case of~\eqref{eq:induction statement general packing} (with our fixed $k$ and) with $0<r\le 2^{k-1} d$ follows as before by taking $p<e^{-d}$ and noting that $s_0$ itself needs to be open for the event $E_{\S, p}(s_0, r)$ to occur. Thus we fix $r>2^{k-1}d$ and assume, as the second induction hypothesis, that~\eqref{eq:induction statement general packing} is established with our fixed $k$ when the $r$ of~\eqref{eq:induction statement general packing} is restricted to $[0,r - 2^k]$.

  Let $(\S, s_0)\in \Sigma^k$. We again reduce to the case in which $\diam(s_0)\in[1, 2^{k-1}]$, if this is not already the case, by a geometric construction. If $\diam(s_0)\in(2^{k-1}, 2^k]$ we let $C_0$ be a cube of side length $2^k$ which contains $s_0$. We partition $C_0$ into $2^d$ sub-cubes $(C_0^i)$ of side length $2^{k-1}$ and set $s_0^i := C_0^i \cap s_0$, where it is noted that some of the $s_0^i$ may be empty and that even if $s_0$ is $\eps$-regular the (non-empty) $s_0^i$ need not be such. Let $I = \{1\le i\le 2^d\colon s_0^i\neq\emptyset\}$. For each $i\in I$ let $(\S^i, s_0^i)\in\Sigma^k$ where $\S^i := (\S\setminus\{s_0\})\cup\{s_0^i\}$. The natural coupling of $G_{\S}^p$ with $G_{\S^i}^p$ is used, where the vertex sets of these (random) graphs are the same except that $s_0\in V(G_{\S}^p)$ if and only if $s_0^i\in V(G_{\S^i}^p)$. Again, $E_{\S,p}(s_0,r)\subseteq\cup_{i\in I} E_{\S^i,p}(s_0^i, r)$ whence
  \begin{equation*}
    \P(E_{\S,p}(s_0,r))\le \sum_{i\in I} \P(E_{\S^i,p}(s_0^i, r)).
  \end{equation*}
  It thus suffices to prove that for each $(\S, s_0)\in \Sigma^k$ with $\diam(s_0)\in[1,2^{k-1}]$ we have
  \begin{equation}\label{eq:goal after dividing s 0 general packing}
    \P(E_{\S,p}(s_0,r))\le \frac{1}{2^d}e^{-\frac{r}{2^{k-1}}}.
  \end{equation}

  Fix such an $(\S, s_0)$. The proof of~\eqref{eq:goal after dividing s 0 general packing} begins exactly as in the proof for square packings, as we briefly recall now. The set $\S$ is partitioned to $\S^{k-1}$ and $\S_{k-1}^k$ which hold, respectively, the sets of diameter in $[1,2^{k-1}]$ and in $(2^{k-1},2^k]$. Correspondingly, we have $E_{\S,p}(s_0, r) = E^{k-1}\cup\bigcup_{s\in \S_{k-1}^k} E_s$ where $E^{k-1}$ is the event that there is an open path in $\S^{k-1}$ from $s_0$ to distance $r$ and, for each $s\in\S_{k-1}^k$, $E_s$ is the event that there is an open path in $\S^{k-1}$ from $s_0$ to a neighbor of $s$ and a disjoint open path in $\S$ from $s$ to distance $r$ from $s_0$. For the purpose of applying the induction hypotheses in the next step we note here that as the two paths are disjoint, the path from $s$ to distance $r$ from $s_0$ cannot contain $s_0$ and thus may be thought of as a path in the \emph{$\eps$-regular packing} $\S\setminus\{s_0\}$. Applying the induction hypotheses and the van-den-Berg--Kesten inequality~\cite{BK85},
  \begin{equation}\label{eq:E S p union bound general packing}
  \begin{split}
    \P(E_{\S,p}(s_0, r))&\le \P(E^{k-1}) + \sum_{s\in \S_{k-1}^k} \P(E_s)\\
     &\le e^{-\frac{r}{2^{k-2}}} + \sum_{s\in\S_{k-1}^k} \min\{p, \alpha(k-1, d(s_0, s)-2^{k-1})\}\cdot\alpha(k, r - d(s_0, s) - 2^k)\\
     &\le e^{-\frac{r}{2^{k-2}}} + \sum_{s\in\S_{k-1}^k} \min\{p, e^{-\frac{d(s_0, s)-2^{k-1}}{2^{k-2}}}\}\cdot e^{-\frac{r - d(s_0, s) - 2^k}{2^{k-1}}},
  \end{split}
  \end{equation}
  where $\alpha(k,r'):=1$ when $r'\le 0$.

  We again decompose $\S_{k-1}^k$ to
  \begin{equation*}
  \begin{split}
    \S_{k-1}^k(0) &:= \{s\in \S_{k-1}^k\colon d(s_0, s)\le m_0\cdot2^k\},\\
    \S_{k-1}^k(m) &:= \{s\in \S_{k-1}^k\colon m 2^k<d(s_0, s)\le (m+1)2^k\}, \quad m\ge m_0
  \end{split}
  \end{equation*}
  with the integer $m_0 = m_0(d,\eps)$ sufficiently large for the following calculations, and proceed to upper bound the size of $\S_{k-1}^k(m)$. By the $\eps$-regularity assumption (and the fact that $s_0\notin\S_{k-1}^k$), each $s\in\S_{k-1}^k$ has $\vol(\partial s)=0$, $\vol(s)\ge\eps 2^{(k-1)d}$ and is fully contained in the $\ell^\infty$ ball $B_s$ of radius $2^{k-2}+d(s_0, s)+2^k$ around the center of $s_0$. As
  \begin{equation*}
    \vol(B_s) = (2^{k-2}+d(s_0, s)+2^k)^d
  \end{equation*}
  we conclude that
  \begin{equation}\label{eq:vol estimates general packing}
    \begin{split}
      |\S_{k-1}^k(0)|&\le \frac{((m_0+2)2^k)^d}{\eps 2^{(k-1)d}} = \frac{1}{\eps}(2m_0+4)^d,\\
      |\S_{k-1}^k(m)|&\le \frac{((m+3)2^k)^d}{\eps 2^{(k-1)d}} = \frac{1}{\eps}(2m+6)^d,\quad m\ge m_0.
    \end{split}
  \end{equation}
  These bounds may be used in~\eqref{eq:E S p union bound general packing} to obtain
  \begin{equation}\label{eq:final bound general packing}
  \begin{split}
    \P(E_{\S,p}(s_0, r)) &\le e^{-\frac{r}{2^{k-2}}} + \frac{1}{\eps}(2m_0+4)^d\cdot p\cdot e^{-\frac{r}{2^{k-1}}+2m_0+2} + \sum_{m=m_0}^\infty \frac{1}{\eps}(2m+6)^d e^{-\frac{r}{2^{k-1}} - 2m + 4}\\
    &= \frac{1}{2^d}e^{-\frac{r}{2^{k-1}}}(2^d e^{-\frac{r}{2^{k-1}}} + \frac{1}{\eps}(4m_0+8)^d\cdot e^{2m_0+2}\cdot p + \frac{e^4}{\eps}\sum_{m=m_0}^\infty (4m+12)^d e^{-2m}).
  \end{split}
  \end{equation}
  Recalling that $r>2^{k-1}d$ one checks that the expression inside the parenthesis in~\eqref{eq:final bound general packing} can be made smaller than $1$ by choosing first $m_0 = m_0(d,\eps)$ sufficiently large and then $p = p(d,\eps)$ sufficiently small. This verifies~\eqref{eq:goal after dividing s 0 general packing} and finishes the proof of~\eqref{eq:induction statement general packing}.

  The deduction of~\eqref{eq:bounded above diameters general packing} and~\eqref{eq:no open path to infinity general packing} from~\eqref{eq:induction statement general packing} now works in the same way as in the proof for square packings, where, in deducing~\eqref{eq:no open path to infinity general packing}, one relies on volume estimates analogous to~\eqref{eq:vol estimates general packing}.

\end{proof}

\section{No accumulation points}\label{sec:proof of corollary no accumulation points}
In this section we prove Corollary~\ref{cor:circle packing accumulation points}.

Let $\S$ be a locally finite circle packing with at most countably many accumulation points in $\R^2$. First, suppose that $G_\S$ has an infinite connected component. By K{\"o}nig's lemma~\cite{K27} there is an infinite sequence of distinct disks $s_0, s_1,\ldots$ in $\S$ such that $s_n$ is tangent to $s_{n+1}$ for $n\ge 0$. Write $o_j$ for the center of $s_j$. It follows that either
\begin{enumerate}
  \item\label{item:converges to infinity} $(o_n)$ converges to infinity, implying that $d(s_0, s_n)\to\infty$, or
  \item\label{item:converges to accumulation point} $(o_n)$ converges to a point of $\R^2$, which must then be an accumulation point of $\S$.
\end{enumerate}
Let us prove that these are indeed the only alternatives. If $(o_n)$ does not converge to a point of $\R^2$ or to infinity then it has two distinct subsequential limit points $x,y$, one of which may be infinity. Suppose $x$ is not infinity. As $\S$ has at most countably many accumulation points, there exists a radius $r<\|x-y\|_2$ such that the circle $S(x,r):=\{z\in\R^2\colon \|x-z\|_2=r\}$ contains no accumulation point of $\S$. Thus $S(x,r)$ can intersect at most finitely many disks of $\S$. In particular, since the disks $(s_n)$ are distinct, we must have that the $(s_n)$ are completely inside or completely outside $S(x,r)$ starting from some $n$, but this contradicts the fact that both $x$ and $y$ are subsequential limit points for $(o_n)$.

Second, let $p>0$ be the constant from Theorem~\ref{thm:percolation for circle packing}. By way of contradiction, suppose that $G^p_{\S}$ has positive probability to contain an infinite connected component. As $G^p_{\S}$ is locally finite with at most countably many accumulation points we may apply the above arguments to it. Thus there exists, with positive probability, a sequence of distinct open disks $s_0,s_1,\ldots$ with $s_n$ tangent to $s_{n+1}$ such that one of the alternatives above holds. Alternative~\eqref{item:converges to infinity} is ruled out by Theorem~\ref{thm:percolation for circle packing}, applying it to each of the countably many possible starting disks $s_0$. Thus alternative~\eqref{item:converges to accumulation point} must hold for a random accumulation point $q$. As there are at most countably many accumulation points, alternative~\eqref{item:converges to accumulation point} holds with positive probability for some \emph{deterministic} accumulation point~$q_0$. Note that $q_0$ may be on the boundary of at most two disks of $\S$ and we denote by $\S'$ the circle packing $\S$ with these disks removed (and naturally couple $G^p_\S$ with $G^p_{\S'}$). We may assume, without loss of generality, that $(s_n)$ does not contain any of the removed disks. Lastly, we apply a M\"obius transformation $T$ to $\S'$ which sends $q_0$ to infinity. The transformation maps $\S'$ to a new circle packing (as $q_0$ is not on the boundary of any disk) and defines a coupling of $G^p_{\S'}$ with $G^p_{T(\S')}$. Under this coupling, on the event that the centers of the disks $(s_n)$ converge to $q_0$, we have $d(T(s_0), T(s_n))\to \infty$. This possibility, however, has probability zero as explained when discussing alternative~\eqref{item:converges to infinity}, yielding the required contradiction.

\section{Existence of infinite connected components and crossings}\label{sec:existence of crossings}
In this section we prove Lemma~\ref{lem:circle packings of unit disk}, Lemma~\ref{lem:circle packings with slowly growing radii} and Corollary~\ref{cor:square tiling of square}.

\subsection{Proof of Lemma~\ref{lem:circle packings of unit disk}}

Let $p>0$ be the constant from Theorem~\ref{thm:percolation for circle packing}, let $\S$ be a circle packing representing a triangulation and having carrier $\mathbb{D}$ and let $E$ be the event that $G_\S^{1-p}$ does not have an infinite connected component. As $E$ is a tail event it suffices, by Kolmogorov's zero-one law, to show that $\P(E)<1$. Let $\mathcal{P}$ be the collection of infinite paths in $G_\S$ which contain a disk intersecting~$\frac{1}{2}\mathbb{D}=\{z\in\R^2\colon \|z\|<\frac{1}{2}\}$. We will prove that the probability that there exists a path in $\mathcal{P}$ consisting of open disks in $G_\S^{1-p}$ is positive.

Let $\mathcal{C}$ be the collection of cycles in $G_\S$ which intersect all paths in $\mathcal{P}$ (cycles which surround $\frac{1}{2}\mathbb{D}$).
As $\S$ represents a triangulation and has carrier $\mathbb{D}$ it follows that there exists a path in $\mathcal{P}$ whose disks are open in $G_\S^{1-p}$ if and only if there does not exist a cycle in $\mathcal{C}$ whose disks are closed in $G_\S^{1-p}$. For a disk $s\in\S$, define the event
\begin{equation}\label{eq:E s def}
  E_s:=\left\{\textstyle
    \begin{array}{c}
      \text{there exists $C\in\mathcal{C}$ whose disks are closed in $G_\S^{1-p}$}\\
      \text{such that $s\in C$ and $s$ has maximal radius among the disks of $C$}
    \end{array}\right\}.
\end{equation}
Our reasoning so far shows that
\begin{equation}\label{eq:E contained in union of E s}
  E \subset \cup_{s\in \S} E_s.
\end{equation}
Let $\S_k$ be the set of disks in $\S$ with radius in $(2^{-(k+1)}, 2^{-k}]$. As the disks of $\S$ are contained in $\mathbb{D}$ it follows that $\S = \cup_{k=0}^\infty \S_k$. Moreover, by area considerations,
\begin{equation}\label{eq:size of S k}
  |\S_k| \le 2^{2(k+1)},\quad k\ge 0.
\end{equation}
Let $F$ be the event that all disks in the collection $\cup_{k=0}^{5}\S_k$ are open in $G_\S^{1-p}$; as there is a finite number of such disks, the probability of $F$ is positive. Thus, to prove that $\P(E)<1$ it suffices to show that $\P(E\,|\,F)<1$. Taking into account~\eqref{eq:E contained in union of E s} we have
\begin{equation}\label{eq:probability of E given F}
  \P(E\,|\,F) \le \sum_{s\in \S} \P(E_s\,|\,F).
\end{equation}
Our definitions imply that $\P(E_s\,|\,F) = 0$ when $s\in\cup_{k=0}^{5}\S_k$ and that $E_s$ is independent of $F$ when $s\in\cup_{k=6}^{\infty}\S_k$. Hence,
\begin{equation}\label{eq:probability of E s given F}
  \sum_{s\in \S} \P(E_s\,|\,F)\le \sum_{k=6}^{\infty}\sum_{s\in\S_k}\P(E_s).
\end{equation}
Fix $k\ge 6$ and $s\in \S_k$. Observe that any cycle $C\in\mathcal{C}$ containing $s$ as a disk of maximal radius must also contain a disk $s'$ with $d(s,s')\ge \frac{1}{2}$. Note also that the probability for a disk to be closed is the parameter $p$ of Theorem~\ref{thm:percolation for circle packing}. Applying that theorem to the circle packing $\cup_{m=k}^{\infty}\S_m$ with $s_0 = s$ thus implies that
\begin{equation*}
  \P(E_s)\le \exp(-2^{k-1}).
\end{equation*}
Combining this estimate with~\eqref{eq:size of S k},\eqref{eq:probability of E given F} and~\eqref{eq:probability of E s given F} yields
\begin{equation*}
  \P(E\,|\,F)\le \sum_{s\in \S} \P(E_s\,|\,F)\le \sum_{k=6}^{\infty}2^{2(k+1)}\exp(-2^{k-1})<1
\end{equation*}
and thus also $\P(E)<1$, as we wanted to prove.

\subsection{Proof of Lemma~\ref{lem:circle packings with slowly growing radii}}\label{sec:circle packings carried by R2} The lemma is proved in a similar manner to Lemma~\ref{lem:circle packings of unit disk} so we will be brief on some of the details.

Let $p>0$ be the constant from Theorem~\ref{thm:percolation for circle packing}. Let $\S$ be a circle packing representing a triangulation whose carrier is $\R^2$ and satisfying the assumption~\eqref{eq:radius growth} for some $c>0$ sufficiently small for the following arguments. Let $r_0\ge 3$ be large enough so that $f(r)\le 2c\, r / \log\log r$ for all $r\ge r_0$. Let $E$ be the event that $G_\S^{1-p}$ does not have an infinite connected component, so that our goal is to show that $\P(E)<1$.

Let $\mathcal{P}$ be the collection of infinite paths in $G_\S$ which contain a disk intersecting~$r_0\mathbb{D}$. Let $\mathcal{C}$ be the collection of cycles in $G_\S$ which intersect all paths in $\mathcal{P}$ (cycles which surround $r_0\mathbb{D}$).
As $\S$ represents a triangulation carried by $\R^2$ it follows that there exists a path in $\mathcal{P}$ whose disks are open in $G_\S^{1-p}$ if and only if there does not exist a cycle in $\mathcal{C}$ whose disks are closed in $G_\S^{1-p}$. For $s\in \S$ define the event $E_s$ by~\eqref{eq:E s def}, so that the relation~\eqref{eq:E contained in union of E s} again holds.

Let $\S_0$ be the set of disks intersecting $r_0\mathbb{D}$. Let $\S_{m,k}$ be the set of disks in $\S$ whose radius is in $(2^{m-k-1}r_0/\log(m+1),\, 2^{m-k} r_0/\log(m+1)]$ and which intersect the annulus $2^m r_0\mathbb{D}\setminus 2^{m-1} r_0\mathbb{D}$. For each $m\ge 1$, the definition of $r_0$ implies that $\S_{m,k}$ is empty for $k< C$ for some large integer $C$ depending only on $c$ and tending to infinity as $c$ tends to zero. Thus $\S$ is the union of $\S_0$ and the sets $\S_{m,k}$ with $m\ge 1$ and $k\ge C$. In addition, area considerations imply that
\begin{equation}\label{eq:S k m bound}
  |\S_{m,k}| \le 2^{2(k+2)}\log(m+1)^2.
\end{equation}


Let $F$ be the event that all disks in $\S_0$ are open. As $\S_0$ is finite, the probability of $F$ is positive whence it is enough to show that $\P(E\,|\,F)<1$. We now have
\begin{equation}\label{eq:E F big sum}
  \P(E\,|\,F) \le \sum_{m=1}^\infty\sum_{k=C}^\infty \sum_{s\in \S_{m,k}}\P(E_s \,|\, F).
\end{equation}
Let $m\ge 1, k\ge C$ and $s\in\S_{m,k}$. Observe that, as $C$ is large, any cycle $C\in\mathcal{C}$ containing $s$ as a disk of maximal radius must also contain a disk $s'$ with $d(s,s')\ge 2^{m-2}r_0$. Thus, applying Theorem~\ref{thm:percolation for circle packing}, with $s_0 = s$ and the circle packing consisting of the disks in $\S$ whose radius is at most that of $s$ and which are not in $\S_0$, implies that
\begin{equation}\label{eq:E s F bound}
  \P(E_s\,|\,F)\le \exp(-2^{k-2}\log(m+1)).
\end{equation}
Combining~\eqref{eq:S k m bound},~\eqref{eq:E F big sum} and~\eqref{eq:E s F bound} we conclude that
\begin{equation*}
  \P(E \,|\, F) \le \sum_{m=1}^\infty\sum_{k=C}^\infty 2^{2(k+2)}\log(m+1)^2 \exp(-2^{k-2} \log(m+1))<1
\end{equation*}
when $C$ is a sufficiently large absolute constant, as we wanted to prove.

%

\subsection{Proof of Corollary~\ref{cor:square tiling of square}} The proof is again a variation on the proof of Lemma~\ref{lem:circle packings of unit disk}.

Let $p>0$ be the constant from Theorem~\ref{thm:percolation for general packings} for a packing of squares in $\R^2$. Let $\S$ be a packing of finitely many squares in the unit square. Let $E$ be the event of a top-bottom crossing of the unit square in $G_\S^p$. We need to bound the probability of $E$ away from one uniformly in $\S$.

Let $\mathcal{C}$ be the collection of paths in $G_\S$ which connect the top of the unit square to its bottom. For a square $s\in\S$, define the event
\begin{equation*}
  E_s:=\left\{\textstyle
    \begin{array}{c}
      \text{there exists $C\in\mathcal{C}$ whose squares are open in $G_\S^p$}\\
      \text{such that $s\in C$ and $s$ has maximal diameter among the squares of $C$}
    \end{array}\right\}.
\end{equation*}
Let $\S_k$ be the set of squares in $\S$ with diameter in $(2^{-(k+1)}, 2^{-k}]$. By area considerations,
\begin{equation}\label{eq:bound on number of squares}
  |\S_k| \le 2^{2(k+1)}.
\end{equation}
Let $F$ be the event that all squares in $\cup_{k=0}^5\S_k$ are closed. As the number of squares in this collection is at most an absolute constant by~\eqref{eq:bound on number of squares}, the probability of $F$ is at least an absolute constant. Thus it suffices to bound $\P(E\,|\,F)$ from one, uniformly in $\S$. We use that
\begin{equation}\label{eq:P E F squares}
  \P(E\,|\,F)\le \sum_{k=0}^\infty \sum_{s\in\S_k} \P(E_s\,|\, F) = \sum_{k=6}^\infty \sum_{s\in\S_k} \P(E_s).
\end{equation}
Lastly, if $s\in\S_k$ for some $k\ge 6$ then any cycle $C\in\mathcal{C}$ containing $s$ as a square of maximal diameter must contain another square $s'\in\S$ with $d(s,s')\ge \frac{1}{4}$. Thus, combining Theorem~\ref{thm:percolation for general packings} (for each $s\in\S_k$ it is applied with $s = s_0$ and the square packing $\cup_{m=k}^\infty\S_m$) and the bounds~\eqref{eq:bound on number of squares} and~\eqref{eq:P E F squares} leads to
\begin{equation*}
  \P(E\,|\,F)\le \sum_{k=6}^\infty 2^{2(k+1)}\exp(-2^{k-2})
\end{equation*}
which is bounded away from one uniformly in $\S$, as we wanted to prove.


\section{Circle packings of Benjamini--Schramm limits of finite planar graphs}
In this section we prove Lemma~\ref{lem:circle packing with at most one accumulation point}. The proof is encumbered by technical details so for the reader's convenience we briefly describe its steps here: (i) Embed each $G_n$ inside a finite \emph{triangulation} $\bar{G}_n$ in a way that the degrees in $\bar{G}_n$ of the vertices of $G_n$ are controlled. (ii) Use the circle packing theorem to obtain a circle packing of each $\bar{G}_n$. Translate and dilate the obtained circle packing so that a uniformly chosen root disk becomes the unit disk. (iii) For each $r$, use the assumption of Benjamini--Schramm convergence and the ring lemma (as $\bar{G}_n$ is a triangulation) to obtain tightness for the centers and radii of the disks of the circle packing which correspond to vertices of $G_n$ and are at graph distance at most $r$ from the root disk. (iv) Use the obtained tightness to extract, along a subsequence, a limiting circle packing of the disks corresponding to vertices of $G_n$. (v) The magic lemma of Benjamini--Schramm implies that the limiting circle packing has at most one accumulation point. (vi) Argue that the limiting circle packing represents the graph $G$ (for this last step we place additional restrictions on the triangulations $(\bar{G}_n)$).

\subsection{Ingredients}
We start by describing several tools which are used in the proof.
\subsubsection{Convergence of circle packings}\label{sec:convergence of circle packings}
Call a circle packing $\S$ \emph{connected} (\emph{locally finite}) if its underlying graph $G_\S$ is connected (locally finite). A rooted circle packing is a pair $(\S, s)$ with $\S$ a circle packing and $s$ one of the disks in $\S$. For a locally finite rooted circle packing $(\S,s)$ and integer $r\ge 0$, write $B_\S(s,r)$ for the induced subgraph of $G_\S$ on the disks whose graph distance to $s$ is at most $r$, rooted at $s$. We proceed to define a notion of convergence for a sequence of rooted circle packings.

Let $(\S_n, s_n)$ be a sequence of finite, connected, rooted circle packings. Say that $(\S_n, s_n)$ \emph{converge locally} to the triple $((G,\rho), \S, \tau)$, where $(G,\rho)$ is a locally finite connected rooted graph, $\S$ is a circle packing and $\tau$ is a bijection of the vertices of $G$ and the disks of $\S$, if
\begin{enumerate}
  \item (local graph convergence) There is a non-decreasing sequence of integers $(r_n)$ tending to infinity such that for each $n$ there exists an isomorphism $I_n$ between the rooted graphs $B_{\S_n}(s_n, r_n)$ and $B_G(\rho, r_n)$.
  \item (convergence of disks) For each vertex $v$ in $G$, the disk $I_n^{-1}(v)$ (well defined for large $n$) converges as $n$ tends to infinity to a non-trivial disk (i.e., its center converges to a point in $\R^2$ and its radius converges to a number in $(0,\infty)$). The mapping $\tau$ takes $v$ to the limiting disk and $\S$ is the set of all $\tau(v)$, $v\in G$ (it is straightforward that $\S$ is a circle packing).
\end{enumerate}
It is pointed out that this convergence does not force the graph $G_\S$ of the circle packing $\S$ to coincide with $G$. Indeed, adjacency of $v$ to $w$ in $G$ implies that $\tau(v)$ is tangent to $\tau(w)$, but the converse implication may fail in general. Moreover, the graph $G_\S$ may even fail to be locally finite (e.g., it is possible that $G_{\S_n}$ is a path of length $n$ and $G_\S$ is a star graph with infinitely many `arms' of length $1$). In general, we thus have only  that $(G_\S,\tau(\rho))$ contains $(G,\rho)$ as a rooted subgraph.

We now extend the above convergence notion to \emph{random} rooted circle packings. A sequence of random finite, connected, rooted circle packings $(\S_n, s_n)$ is said to \emph{converge locally in distribution} to the random triple $((G,\rho), \S, \tau)$ with $(G,\rho), \S, \tau$ random objects of the above types if there exists a coupling of the $(\S_n, s_n)$ and $((G,\rho), \S, \tau))$ so that, almost surely under this coupling, $(\S_n, s_n)$ converges locally to $((G,\rho), \S, \tau)$.

The following \emph{tightness condition} will be of use. Let $(\S_n, s_n)$ be a sequence of random, finite, connected rooted circle packings. Let $(G,\rho)$ be a random locally finite connected rooted graph $(G,\rho)$. Suppose that (i) for each integer $r\ge 0$ and each rooted graph $(H,\sigma)$ it holds that $\P(B_{\S_n}(s_n, r)=(H,\sigma))\to \P(B_G(\rho, r)=(H,\sigma))$, (ii) for each $r\ge 0$, the centers of the disks in all $B_{\S_n}(s_n, r)$, as $n$ varies, are tight, when considered as random variables in $\R^2$, and (iii) for each $r\ge 0$, the radii of the disks in all $B_{\S_n}(s_n, r)$, as $n$ varies, are tight, when considered as random variables in $(0,\infty)$. Then, via compactness arguments (and the Skorohod representation theorem), there exists a subsequence $(n_k)$, random circle packing $\S$ and random bijection $\tau$ so that $(\S_{n_k}, s_{n_k})$ converges locally in distribution to $((G,\rho), \S, \tau)$.

\subsubsection{Applying the magic lemma}

Benjamini and Schramm proved that every Benjamini–Schramm limit of finite simple planar graphs with uniformly bounded degrees may almost surely be realized as the tangency graph of a circle packing with at most one accumulation point. The main tool in their proof is the so-called \emph{magic lemma}~\cite[Lemma 2.3]{BS01}. We now apply the magic lemma to obtain a similar conclusion in our context.

Denote the closed unit disk by $\bar{\mathbb{D}}=\{z\in\R^2\colon \|z\|\le 1\}$.
\begin{lemma}(Benjamini--Schramm limit of finite circle packings)\label{lem:magic lemma consequence}
Let $(\S_n)$ be a sequence of random finite and connected circle packings. For each $n$, let $s_n$ be a uniformly sampled disk in $\S_n$ (again, first $\S_n$ is sampled and then $s_n$ is sampled uniformly from it). Let $T_n$ be the (unique) mapping $z\mapsto az+b$ with $a>0$ and $b\in\R^2$ for which $T_n(s_n) = \bar{\mathbb{D}}$. If $(T_n(\S_n), \bar{\mathbb{D}})$ converges locally in distribution to some $((G,\rho), \S, \tau)$ then the circle packing $\S$ has at most one accumulation point in $\R^2$, almost surely.
\end{lemma}
\begin{proof}
   The proof follows that of~\cite[Proposition 2.2]{BS01}.
\end{proof}

\subsubsection{The circle packing theorem}
To make use of Lemma~\ref{lem:magic lemma consequence} we need a way to generate appropriate circle packings. This is provided by the following celebrated theorem of Koebe.
\begin{thm}(The circle packing theorem~\cite{K36},~\cite[Theorem 3.5]{N18})\label{thm:circle packing}
  For any finite planar map $G$ there exists a circle packing $\S$ which represents $G$.
\end{thm}

\subsubsection{The ring lemma}
The following lemma of Rodin and Sullivan will be used to check the tightness condition discussed in Section~\ref{sec:convergence of circle packings}.

In a circle packing, we say that the disks $s_1,\ldots, s_M$ \emph{completely surround} the disk $s_0$ if each $s_i$ is tangent to $s_0$ and $s_i$ is tangent to $s_{i+1}$ for $1\le i\le M$, where we set $s_{M+1}:=s_1$.
\begin{lemma}(The ring lemma~\cite{RS87},~\cite[Lemma 4.2]{N18})\label{lem:ring lemma}
For each integer $M>0$ there exists $c(M)>0$ for which the following holds. Suppose that $s_0, \ldots, s_M$ are disks in a circle packing and $s_1,\ldots, s_M$ completely surround $s_0$. Let $r_i$ be the radius of $s_i$. Then $r_i / r_0 \ge c(M)$ for all $i$.
\end{lemma}

\subsubsection{Extension to a triangulation}
To verify the assumption of the ring lemma we will need to work with triangulations. The following lemma provides a way to extend a finite planar graph to a triangulation with control on the degrees of the vertices of the original graph as well as their new neighbors.
\begin{lemma}\label{lem:extension to triangulation}
  Let $H$ be a finite simple connected planar graph. There exists a finite simple \emph{triangulation} $\bar{H}$ such that $H$ is contained in the graph of $\bar{H}$, $\deg_{\bar{H}}(v)=3\deg_H(v)$ for all vertices $v$ of $H$ (where $\deg_G(v)$ is the degree of $v$ in the graph $G$) and $\deg_{\bar{H}}(v)=5$ for all vertices in $\bar{H}\setminus H$ which are neighbors of the vertices of $H$.
\end{lemma}
\begin{proof}
  Draw $H$ in the plane with straight lines for edges (e.g., with a circle packing) and consider $H$ with the resulting map structure. Each face $f$ of $H$ may be represented by a directed path $\vec{e}_1^f, \ldots, \vec{e}_{k(f)}^f$, with $k(f)\ge 3$ as $H$ is simple. Write $\vec{e}_j^f = (v_j^f, v_{j+1}^f)$ so that $v_1^f, \ldots, v_{k(f)}^f, v_{k(f)+1}^f=v_1^f$ is a cycle in which the vertices need not be distinct. For each face $f$: First, draw a new cycle $w_1^f, \ldots, w_{k(f)}^f,  w_{k(f)+1}^f = w_1^f$ having distinct vertices in the region surrounded by $(v_j^f)$ in the drawing (the bounded region, unless $f$ is the `outer face'), and draw edges from $w_j^f$ to both $v_j^f$ and $v_{j+1}^f$ for $1\le j\le k$. Second, add a new vertex $u^f$ in the region surrounded by the cycle $(w_j^f)$ in the drawing and draw edges from $u^f$ to each of the $(w_j^f)$. The resulting proper drawing describes a simple triangulation which contains $H$ as a subgraph, and it is straightforward to check that the degree of each vertex of $H$ is exactly tripled in this construction. In addition, among the added vertices, the only ones neighboring those of $H$ are the $(w_j^f)$ and each of these has degree $5$ by construction.
\end{proof}

\subsection{Proof of Lemma~\ref{lem:circle packing with at most one accumulation point}}
Let $(G_n)$ be a sequence of, possibly random, finite simple planar graphs with Benjamini--Schramm limit $(G,\rho)$. Let $\rho_n$ be uniformly sampled in $G_n$.

\subsubsection{First observations} We begin with some observations. First we may, and will, assume without loss of generality that each $G_n$ is connected. Indeed, otherwise we may replace $G_n$ with $\hat{G}_n$ where the distribution of $\hat{G}_n$ is obtained by first sampling $G_n$ and then sampling any one of the connected components of $G_n$ with probability proportional to the number of vertices in the component. It is simple to check that for any $r\ge 0$, the distribution of $B_{G_n}(\rho_n, r)$ is equal to that of $B_{\hat{G}_n}(\hat{\rho}_n, r)$, where $\hat{\rho}_n$ is uniformly sampled in $\hat{G}_n$. Thus $(G,\rho)$ is also the Benjamini--Schramm limit of $(\hat{G}_n)$.

Second, write $|H|$ for the number of vertices of a graph $H$. We assume without loss of generality that $|G|=\infty$, almost surely.
Let us show that this is indeed without loss of generality. The assumption that $(G_n)$ converges to $(G,\rho)$ in the Benjamini--Schramm sense implies that there exists a coupling of $((G_n, \rho_n))_n$ and $(G_,\rho)$ so that for each integer $r\ge 0$, $B_{G_n}(\rho_n, r)$ converges to $B_G(\rho, r)$ almost surely (by the Skorohod representation theorem). Under the coupling, the event $\{|G|=\infty\}$ equals the event $\{\lim_n |G_n|=\infty\}$ (as $(G_n)$ are connected). On the event $\{|G|<\infty\}$ it is clear that $G$ can be represented by a circle packing with no accumulation points, by Theorem~\ref{thm:circle packing}. Assume that the probability of $\{|G|=\infty\}$ is positive, as otherwise the lemma follows trivially. Condition (under the coupling) on $\{|G|=\infty\}$ and note that it still holds under the conditioning that $\rho_n$ is uniformly distributed in $G_n$, as $\{|G|=\infty\}$ is independent of the choice of the $(\rho_n)$ in $(G_n)$ (it depends only on $|G_n|$, as explained above). We may now replace the distribution of each $G_n$ and the distribution of $(G,\rho)$ by their distribution conditioned on $\{|G|=\infty\}$ and preserve the property that $(G_n)$ converges in the Benjamini--Schramm sense to $(G,\rho)$, while adding the property that $|G|=\infty$ almost surely.

Third, the assumption that $|G|=\infty$ almost surely implies that $|G_n|$ converges to infinity in probability, i.e., that
\begin{equation}\label{eq:size of G_n tends to infinity}
\text{For each $M>0$,\quad $\P(|G_n|\le M)\to 0$\quad as $n\to\infty$}.
\end{equation}
In addition, the Benjamini--Schramm convergence of $(G_n)$ to $(G,\rho)$ implies that for each integer $r\ge 0$, the number of vertices of $G_n$ at distance at most $r$ from $\rho_n$ is tight as $n$ tends to infinity. In other words,
\begin{equation}\label{eq:tightness of size of r neighborhoods}
  \text{For each integer $r\ge 0$,\quad $\lim_{M\to\infty}\sup_{n\to\infty}\P(|B_{G_n}(\rho_n, r)|\ge M) = 0$}.
\end{equation}
Combining~\eqref{eq:size of G_n tends to infinity} and~\eqref{eq:tightness of size of r neighborhoods} we conclude that for each random sequence of vertices $v_n\in G_n$ with $v_n$ independent of $\rho_n$, the distance between $v_n$ and $\rho_n$ tends to infinity in probability, i.e.,
\begin{equation}\label{eq:v n out of neighborhood}
  \text{For each integer $r\ge 0$,\quad $\P(v_n\in B_{G_n}(\rho_n, r))\to 0$ as $n\to\infty$}.
\end{equation}
To see this, observe that the uniformity of $\rho_n$ (and its independence from $v_n$) implies that, for each integer $r\ge 0, M\ge 1$,
\begin{equation*}
\begin{split}
  \P(v_n\in B_{G_n}(\rho_n, r)\, |\, G_n, v_n) &= \P(\{v_n\in B_{G_n}(\rho_n, r)\}\cap\{|B_{G_n}(v_n, r)|\ge M\}\, |\, G_n, v_n)\\
   &+ \P(\{v_n\in B_{G_n}(\rho_n, r)\}\cap\{|B_{G_n}(v_n, r)|< M\}\, |\, G_n, v_n)\\
   &\le \P(|B_{G_n}(\rho_n, 2r)|\ge M\, |\, G_n, v_n) + \min\left\{\frac{M}{|G_n|}, 1\right\}.
\end{split}
\end{equation*}
Averaging over $G_n$ and $v_n$ we see that both terms can be made as small as we like by choosing first $M$ large and then $n$ large, by~\eqref{eq:tightness of size of r neighborhoods} and~\eqref{eq:size of G_n tends to infinity}.

\subsubsection{Circle packings}
We proceed to make use of the circle packing theorem. For each $n$, let $\bar{G}_n$ be the extension of $G_n$ to a (random) finite simple triangulation given by Lemma~\ref{lem:extension to triangulation}. Apply the circle packing theorem, Theorem~\ref{thm:circle packing}, to $\bar{G}_n$ to obtain a (random) circle packing $\bar{S}_n$. Let $\S_n$ be the subset of $\bar{S}_n$ of the disks corresponding to $G_n$ (chosen arbitrarily if there is more than one correspondence). Let $s_n$ be uniformly sampled from the disks of $\S_n$. Let $T_n$ be the unique mapping $z\mapsto az+b$, with $a>0$ and $b\in\R^2$, for which $T_n(s_n) = \bar{\mathbb{D}}$. We will establish that
\begin{enumerate}
  \item There is a subsequence $(n_k)$ for which $(T_{n_k}(\S_{n_k}), \bar{\mathbb{D}})$ converges locally in distribution to $((G,\rho), \S, \tau)$ with $(G,\rho)$ the Benjamini--Schramm limit of $(G_n)$, $\S$ a (random) circle packing and $\tau$ a bijection of the vertices of $G$ and the disks of $\S$.
  \item Almost surely, $G = G_\S$.
\end{enumerate}
The two statements suffice to finish the proof as Lemma~\ref{lem:magic lemma consequence} implies that $\S$ has at most one accumulation point in $\R^2$.

\subsubsection{Tightness}
We first prove the existence of $(n_k)$ for which $(T_{n_k}(\S_{n_k}), \bar{\mathbb{D}})$ converges locally in distribution. We employ the tightness condition discussed in Section~\ref{sec:convergence of circle packings}. For brevity, write $B_{n,r}:=B_{T_n(\S_n)}(\bar{\mathbb{D}}, r)$.

The fact that for each integer $r\ge 0$ and each rooted graph $(H,\sigma)$ the convergence $\P(B_{n,r}=(H,\sigma))\to \P(B_G(\rho, r)=(H,\sigma))$ holds is equivalent to our assumption that $(G,\rho)$ is the Benjamini--Schramm limit of $(G_n)$.

To see that for each $r$ the centers and radii of the disks in $(B_{n,r})$, as $n$ varies, are tight we make use of the ring lemma, Lemma~\ref{lem:ring lemma}, applied to the larger circle packing $\bar{\S}_n$. Fix an integer $r\ge 0$. As $\bar{G}_n$ is a triangulation, every disk in $T_n(\bar{\S}_n)$, besides three disks $b_n^1, b_n^2, b_n^3$ (which border the `outer face'), is completely surrounded by other disks in $T_n(\bar{\S}_n)$. The three disks, or a subset of them, may belong to the smaller circle packing $\S_n$ but in any case they are unlikely to be in the neighborhood $B_{n,r}$. Precisely, by~\eqref{eq:v n out of neighborhood},
\begin{equation}\label{eq:boundary is far}
  \P(\{b_n^1, b_n^2, b_n^3\}\cap B_{n,r}\neq \emptyset) \to 0\quad\text{as $n\to\infty$}.
\end{equation}
On the event in~\eqref{eq:boundary is far}, the ring lemma may be used for each of the disks in $B_{n,r}$, when considered inside the larger circle packing $T_n(\bar{\S}_n)$, by considering a shortest path in $B_{n,r}$ from the disk to $\bar{\mathbb{D}}$. Write $\bar{\Delta}_{n,r}$ for the maximal degree of the disks of $B_{n,r}$ when considered in the graph $G_{T_n(\bar{\S}_n)}$ and $\Delta_{n,r}$ for the maximal degree of the same disks when considered in the graph $G_{T_n(\S_n)}$. Lemma~\ref{lem:extension to triangulation} shows that $\bar{\Delta}_{n,r} = 3\Delta_{n,r}$. Thus the ring lemma, together with~\eqref{eq:boundary is far}, implies that the centers and radii of the disks in $B_{n,r}$ are tight if $\Delta_{n,r}$ is tight. This latter statement now follows from~\eqref{eq:tightness of size of r neighborhoods}.


\subsubsection{Graph of the limiting circle packing} It remains to prove that $G = G_\S$ almost surely. The discussion in Section~\ref{sec:convergence of circle packings} already shows that $G$ is a subgraph of $G_\S$, so we need only show that if $v,w\in G$ are non-neighbors, then their disks $\tau(v), \tau(w)\in\S$ are non-tangent. This in turn is implied by the following claim: For each integer $r\ge 0$, the minimal distance between two non-tangent disks in $B_{n,r}$ is a tight random variable in $(0,\infty)$, as $n$ tends to infinity. To see the claim, introduce the intermediate circle packing $\hat{\S}_n$, consisting of the disks of $\S_n$ and the disks of $\bar{\S}_n$ which are tangent to them. As explained in the previous section, the disks in $B_{n,r}$ are completely surrounded by disks in $T_n(\hat{\S}_n)$, with probability tending to one as $n$ tends to infinity. These surrounding disks can act as a `barrier', preventing two non-tangent disks of $B_{n,r}$ from coming too near to each other. With this idea, the above claim will follow once we show that the minimal radius of a disk in $B_{T_n(\hat{\S}_n)}(\bar{\mathbb{D}}, r+1)$ is tight in $(0,\infty)$ as $n$ tends to infinity. The proof is similar to that of the previous section, making use of two facts: (i) Each disk in $\hat{\S}_n\setminus \S_n$ is tangent to exactly 5 other disks in $\bar{\S}_n$ by Lemma~\ref{lem:extension to triangulation}, and (ii) if a boundary disk $b_n^1, b_n^2$ or $b_n^3$ belongs to $T_n(\hat{\S}_n\setminus \S_n)$ and is also in $B_{T_n(\hat{\S}_n)}(\bar{\mathbb{D}}, r+1)$ then one of its 5 neighbors belongs to $B_{n,r+2}$, which is unlikely by~\eqref{eq:v n out of neighborhood}.

\section{Discussion, open questions and conjectures}\label{sec:discussion and open questions}
\subsection{Percolation on circle packings} Theorem~\ref{thm:percolation for circle packing} shows that there exists $p>0$ such that for any circle packing, there is zero probability for the origin to be connected to infinity by open disks after site percolation with parameter $p$. What is the largest value of $p$ for which this statement holds? The example of the triangular lattice shows that $p$ cannot exceed $1/2$. The following conjecture states, with the notation of Section~\ref{sec:circle packings}, that this bound is tight.
\begin{conj}\label{conj:p=1/2}
  Let $\S$ be a circle packing and $s_0\in \S$. Then
  \begin{equation}
    \P(E_{\S, 1/2}(s_0, \infty))=0.
  \end{equation}
\end{conj}
The conjecture is similar in spirit to Conjecture~\ref{conj:square tiling} of Benjamini. We emphasize that Conjecture~\ref{conj:p=1/2} does not make any assumptions on the circle packing but point out that it is open and interesting also when the circle packing represents a triangulation, has carrier $\R^2$ and uses disks whose radii are uniformly bounded from zero and infinity.

If the conjecture is verified then it would follow that site percolation with $p=1/2$ on the following classes of graphs has no infinite connected component almost surely (using the same proofs as the corresponding statements here): (i) Graphs represented by locally finite circle packings with at most countably many accumulation points (as in Corollary~\ref{cor:circle packing accumulation points}). (ii) Recurrent simple plane triangulations (as in part~\eqref{item:recurrent plane triangulations} of Corollary~\ref{cor:percolation on plane triangulations}), giving a positive answer to Question~\ref{ques:recurrent triangulations} of Benjamini. (iii) Benjamini--Schramm limits of, possibly random, finite simple planar graphs (as in Corollary~\ref{cor:percolation on Benjamini Schramm limits}).

\subsection{Exponential decay} The second part of Theorem~\ref{thm:percolation for circle packing} states that if the radii of the disks in the circle packing are uniformly bounded above then the probability in the site percolation that the origin is connected to distance $r$ decays exponentially in $r$, at a rate which is uniform in the circle packing. Such a statement cannot hold at $p=1/2$, again due to the example of the triangular lattice for which this value of $p$ is critical. On transitive graphs, exponential decay of connection probabilities has been shown to hold in the sub-critical regime of percolation~\cites{M86,AB87,DCT16,DCRT19}. If Conjecture~\ref{conj:p=1/2} is verified, then it is natural to conjecture also that exponential decay holds for all $p<1/2$, uniformly in circle packings whose disks have bounded above radii.
\begin{conj}\label{conj:uniform exponential decay}
  There exists $f:(0,1/2)\to(0,\infty)$ such that the following holds. Let $0<p<1/2$. Let $\S$ be a circle packing and $s_0\in \S$. Assume that $D:=\sup_{s\in\S}\diam(s)<\infty$. Then for each~$r>0$,
  \begin{equation}\label{eq:conj bounded above circle packing}
    \P(E_{\S, p}(s_0, r))\le \exp\left(-f(p)\frac{r}{D}\right).
  \end{equation}
\end{conj}
The fact that $D$ enters~\eqref{eq:conj bounded above circle packing} via the ratio $\frac{r}{D}$ follows by a simple scaling consideration. Theorem~\ref{thm:percolation for circle packing} verifies Conjecture~\ref{conj:uniform exponential decay} for sufficiently small $p$. Again, the conjecture is open and interesting also when the circle packing represents a triangulation, has carrier $\R^2$ and uses disks whose radii are uniformly bounded from zero.

If Conjecture~\ref{conj:uniform exponential decay} is verified then it would follow that site percolation with any $p>1/2$ on transient simple bounded degree plane triangulations has an infinite connected component almost surely (as in part~\eqref{item:transient plane triangulations} of Corollary~\ref{cor:percolation on plane triangulations}). Proving the existence of an infinite connected component at $p=1/2$ and thus verifying Conjecture~\ref{conj:transient triangulations} of Benjamini requires additional tools. One may similarly deduce an analogue of Conjecture~\ref{conj:square tiling} of Benjamini in which the squares are replaced by circles and the probability to color a circle black is strictly greater than $1/2$.

\subsection{Value of $p_c$ and scaling limit}\label{sec:value of p_c} For well-behaved circle packings, it may be that the critical probability for site percolation is exactly $1/2$. We expect this for circle packings representing a triangulation, having carrier $\R^2$ and using disks whose radii are uniformly bounded from infinity. Indeed, this will follow, together with the fact that there is no infinite connected component at $p=1/2$, from Conjecture~\ref{conj:p=1/2} and Conjecture~\ref{conj:uniform exponential decay} (as in Lemma~\ref{lem:circle packings with slowly growing radii}). For such circle packings, it may even be that the scaling limit of the $p=1/2$ percolation is the Conformal Loop Ensemble~\cites{S09, SW12} with parameter $\kappa=6$ as for the triangular lattice~\cites{S01,CN04,CN05,CN06}. Beffara~\cite{B08} discusses the question of finding embeddings of triangulations on which percolation has a conformal invariant scaling limit and makes a related conjecture~\cite[Conjecture 12]{B08}.

\begin{figure}
		\includegraphics[scale=0.3]{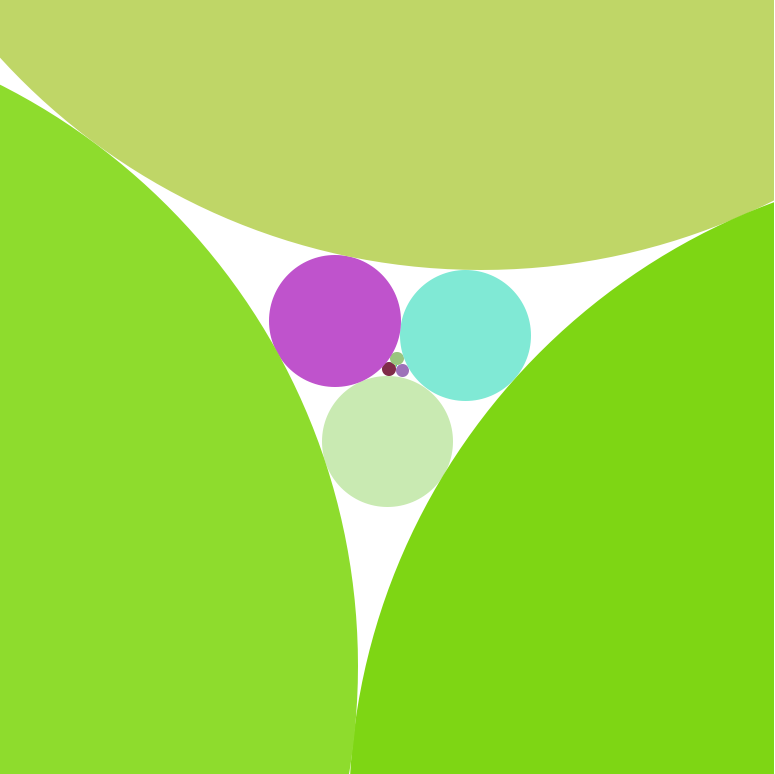}
	\caption{Part of the circle packing of the graph obtained by triangulating the product of a triangle with the natural numbers.}
	\label{fig:circle packing triangle natural numbers}
\end{figure}

We point out that the restriction that the radii are uniformly bounded above cannot be waived without a replacement. Indeed, consider the product of a triangle with the natural numbers, endowed with a natural map structure and triangulated with a diagonal edge added in every face with 4 edges. The circle packing of one such triangulation is depicted in Figure~\ref{fig:circle packing triangle natural numbers}. It is straightforward that is carried by $\R^2$ but uses disks whose radii are unbounded above, growing linearly with their distance to the origin. It is simple to see that the critical probability for site percolation on this graph is one.

\subsection{General shapes and the dependence on the aspect ratio} Theorem~\ref{thm:percolation for general packings} extends Theorem~\ref{thm:percolation for circle packing} from circle packings to $\eps$-regular packings, showing that percolation on the latter class has no infinite connected component at fixed $p$ depending only on $\eps$ and the dimension of the space. Possibly, Conjecture~\ref{conj:p=1/2} and Conjecture~\ref{conj:uniform exponential decay}, apply also for $\eps$-regular packings in $\R^2$, as long as they represent a planar graph (planarity may fail, for instance, for packings of squares in which four squares are allowed to share a corner). In Conjecture~\ref{conj:uniform exponential decay} the rate function $f(p)$ should then be allowed to depend on $\eps$. We state this explicitly for ellipse packings. Define the \emph{aspect ratio} of an elliptical disk as the ratio of lengths of the major axis and the minor axis of its bounding ellipse.
\begin{conj}\label{conj:ellipses}
  There exists $f:(0,1/2)\times[1,\infty)\to(0,\infty)$ such that the following holds. Let $M\ge 1$. Let $\S$ be a packing of closed elliptical disks of aspect ratio at most $M$. Let $s_0\in \S$. Then
  \begin{equation}
    \P(E_{\S, 1/2}(s_0, \infty))=0.
  \end{equation}
  If, in addition, $D:=\sup_{s\in\S}\diam(s)<\infty$, then for each $0<p<1/2$ and each~$r>0$,
  \begin{equation}\label{eq:conj bounded above ellipse packing}
    \P(E_{\S, p}(s_0, r))\le \exp\left(-f(p,M)\frac{r}{D}\right).
  \end{equation}
\end{conj}

\begin{figure}
		\includegraphics[scale=0.25, trim = 0 0 10 0, clip]{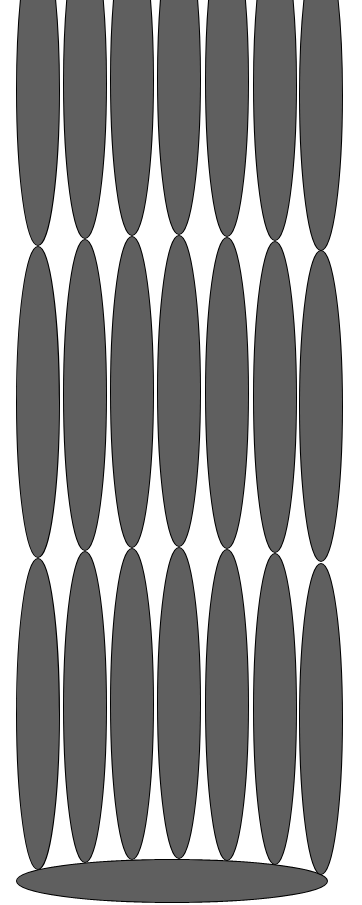}
	\caption{Part of an ellipse packing showing the effect of the aspect ratio on the rate of exponential decay.} 
	\label{fig:ellipse packing}
\end{figure}
The dependence of $f$ on the aspect ratio is unclear, even for small $p$. Consider, for instance, a packing of elliptical disks of fixed diameter $D$ and fixed aspect ratio $M$ formed as follows: Let one ellipse $s_0$ have its major axis on the $y$-axis and at approximately $M$ locations on the top part of $s_0$ start lines of infinitely many ellipses with their major axis parallel to the $x$-axis (see Figure~\ref{fig:ellipse packing}). In this case, it is straightforward that after site percolation with any $p>0$, there is chance approximately $e^{-1} p$ (for $M$ large) that $s_0$ is open and in one of the lines of ellipses above $s_0$, the first $\lfloor\log_{1/p}(M)\rfloor$ ellipses are open. On this event, $E_{\S,p}(s_0, r)$ occurs for $r\approx\lfloor\log_{1/p}(M)\rfloor D$. Thus $f(p,M)\le \frac{\tilde{f}(p)}{\log M}$ for some $\tilde{f}(p)$ and large $M$. Matan Harel~\cite{H19} has found a construction, for rectangle packings of aspect ratio at most $M$, showing that the analogous $f(p,M)$ must satisfy $f(p,M)\le \frac{\tilde{f}(p)}{M^{\eps(p)}}$ for some functions $\tilde{f},\eps$ of $p$ and large $M$. It is interesting to determine whether indeed the $f(p,M)$ depends on $M$ only through an inverse power as in Harel's construction. If proved for more general packings, such a result would be useful in the study of the loop $O(n)$ model~\cite{CGHP20}, leading to quantitative estimates in \emph{finite} volume.

When the ellipse packing represents a recurrent plane triangulation, one may use the He--Schramm theorem, Theorem~\ref{thm:He--Schramm}, to pass to a circle packing representing the same graph and then apply Theorem~\ref{thm:percolation for circle packing} to deduce the absence of infinite connected components after site percolation with the parameter $p$ of that theorem. The obtained value of $p$ is better than what one would obtain by applying Theorem~\ref{thm:percolation for general packings} directly to the ellipse packing, as the latter value depends on the aspect ratio of the elliptical disks. However, even when this route is available, one still requires additional arguments in order to derive an exponential decay estimate such as~\eqref{eq:conj bounded above ellipse packing} (even at the value of $p$ given by Theorem~\ref{thm:percolation for circle packing}), since the mapping of the ellipse packing to the circle packing introduces non-trivial distortions to the underlying metric.

\subsection{Recurrent planar graphs} In part~\eqref{item:recurrent plane triangulations} of corollary~\ref{cor:percolation on plane triangulations} we prove the absence of an infinite connected component for site percolation on simple \emph{recurrent} plane triangulations, at the value of $p$ given by Theorem~\ref{thm:percolation for circle packing}. In this result, unlike its counterpart for transient plane triangulations in part~\eqref{item:transient plane triangulations} of the corollary, the assumption that the map is a triangulation plays a technical role, allowing us to rely on the results of He--Schramm, Theorem~\ref{thm:He--Schramm}. The role of the assumption that the graph is one-ended is also unclear. Is it in fact the case that there is no infinite connected component for site percolation, at some fixed $p>0$, or even at $p=1/2$ along the lines of Conjecture~\ref{conj:p=1/2} and Benjamini's Question~\ref{ques:recurrent triangulations}, for \emph{all planar recurrent graphs}? This follows, with the $p$ of Theorem~\ref{thm:percolation for circle packing}, for subgraphs of simple recurrent \emph{plane triangulations}, by part~\eqref{item:recurrent plane triangulations} of Corollary~\ref{cor:percolation on plane triangulations}.
%
%
%
%
For multiply-ended planar recurrent graphs the following result of Gurel-Gurevich, Nachmias and Souto~\cite{GNS17} is possibly of relevance: Let $\S$ be a circle packing of a bounded degree triangulation $G$. Then $G$ is recurrent if and only if the set of accumulation points of $\S$ is polar (where a set is polar if it is avoided by two-dimensional Brownian motion, almost surely).


\subsection{Other connectivity notions}\label{sec:other connectivity notions}
For a circle packing with accumulation points, one may define notions of connectivity extending the one used here. Specifically, given a circle packing $\S$, parameter $0<p<1$ and points $x,y$ in $\R^2$ one may say that $x$ is connected to $y$ in $G^p_\S$ if $x$ and $y$ belong to the same connected component of the \emph{closure} (in $\R^2$) of the union of open disks in $G^p$. A more restricted possibility is to say that $x$ and $y$ are connected if there is a \emph{path} (in $\R^2$) connecting them in the closure of the union of open disks in $G^p$. Such definitions allow connections between $x$ and $y$ to `pass through' accumulation points of $\S$. Similar definitions may be used to define a connection between $x$ and infinity. The proof of Theorem~\ref{thm:percolation for circle packing} does not apply to these extended notions of connectivity. 

\subsection{Bond percolation} The analogue of Theorem~\ref{thm:percolation for circle packing} fails for \emph{bond} percolation. Precisely, given a circle packing $\S$ and $p\in[0,1]$, the $p$-bond-percolation process is the graph $G_\S^{p,\text{bond}}$ obtained from $G_\S$ by independently retaining each \emph{edge} of the graph with probability $p$ and discarding it with probability $1-p$. For each $p>0$ there exist circle packings $\S$ without accumulation points such that $G_\S^{p,\text{bond}}$ contains an infinite connected component almost surely. Indeed, one may start with a convenient circle packing $\S_0$, such as the periodic circle packing of the square lattice $\Z^2$, and for each integer $M>0$ create a new circle packing $\S_M$ by adding for each pair of tangent circles $s_0, s_1$ in $\S_0$ additional $M$ circles tangent to both $s_0$ and $s_1$. For each $p>0$ one may take $M$ large enough so that the graph $G_{\S_M}^{p,\text{bond}}$ will contain an infinite connected component almost surely.

\subsection{Related conjectures and results} We mention several additional conjectures and results relating to site percolation on planar graphs and the special value $p=1/2$.

\subsubsection{Isoperimetric assumptions}\label{sec:isoperimetric} Benjamini and Schramm~\cite[Conjecture 3]{BS96} conjectured that $p_c(G)\le 1/2$ for (bounded degree) plane triangulations $G$ for which $|\partial A|\ge f(|A|)\log|A|$ for some $f$ growing to infinity and all finite sets of vertices $A$ (see also~\cite[Section 2.1]{B2018}). Moreover, they conjecture that $p_c(G)<1/2$ if $G$ has positive Cheeger constant.

Georgakopoulos and Panagiotis~\cite[Section 11]{GP2018} make progress on this and other conjectures by proving that the \emph{bond} percolation critical probability is at most $1/2$ for the following classes of graphs: (i) locally finite plane triangulations satisfying the above isoperimetric assumption, (ii) bounded degree transient plane triangulations (making progress towards Benjamini's Conjecture~\ref{conj:transient triangulations}), (iii) bounded degree recurrent plane triangulations which can be represented by a circle packing whose disks have radii which are uniformly bounded above. They further improve the upper bound to $1/2-\eps(\Delta(G))$ where $\Delta(G)$ is the maximal degree in $G$ and $\eps(\Delta)\to0$ as $\Delta\to\infty$ (see~\cite[end of Section 11.2]{GP2018}). For the above graphs, they also explain how to prove the bound $p_c(G)\le 1 - \frac{1}{\Delta(G)-1}$ on the \emph{site} percolation critical probability (this bound is related to the fact that $p_c(G)\ge \frac{1}{\Delta(G)-1}$ as mentioned in Section~\ref{sec:introduction}).

\subsubsection{Degree assumptions} Benjamini and Schramm~\cite[Conjecture 7]{BS96} also conjectured that $p_c(G)< 1/2$ for every planar graph $G$ with minimal degree at least $7$, and that there are infinitely many open connected components when $p\in(p_c(G), 1-p_c(G))$.

Angel, Benjamini and Horesh~\cite[Problem 4.2]{ABH18} asked whether $p_c(G)\le 1/2$ for plane triangulations with degrees at least $6$ (and posed related questions on bond percolation and the connective constant).

Haslegrave and Panagiotis~\cite{HP19} resolve the first part of the conjecture of Benjamini and Schramm, proving that $p_c(G)<1/2$ when the minimal degree is at least $7$ (for planar graphs that have proper drawings without accumulation points in a suitable sense). They also make progress towards the conjecture of Angel, Benjamini and Horesh, proving that $p_c(G)\le 2/3$ when the minimal degree is at least $6$. In a third result they prove an upper bound for $p_c(G)$ when the minimal degree is at least $5$ and the minimal face degree is at least $4$.


\subsubsection{Volume growth and non-amenability} Benjamini~\cite[Section 2.1]{B2018} conjectures that $p_c(G)\ge 1/2$ for plane triangulations with polynomial volume growth. It is also conjectured there that $p_C(G)<1/2$ for nonamenable plane triangulations.

\subsubsection{Number of infinite connected components} A conjecture of a different nature of Benjamini and Schramm~\cite[Conjecture 8]{BS96} is that if site percolation with $p=1/2$ on a planar graph has an infinite connected component almost surely, then it has infinitely many infinite connected components almost surely. This is proved~\cite[Theorem 5]{BS96} when the planar graph admits an embedding in $\R^2$ in which the $x$-axis avoids all edges and vertices and every compact set intersects finitely many vertices and edges.

\subsubsection{Translation-invariant and unimodular graphs} Call a planar graph a \emph{planar lattice} if it has an embedding in $\R^2$ which is invariant under a full-rank lattice of translations and satisfying that every compact set intersects only finitely many vertices and edges. On planar lattices there is no infinite connected component for site percolation with $p=1/2$. This follows from the classical results of Aizenman--Kesten--Newman~\cite{AKN87} and Burton--Keane~\cite{BK89} which rule out the existence of more than one infinite connected component and theorems showing that there is no coexistence of unique open and closed infinite connected components~\cite[Theorem 14.3]{HJ06},\cite[Corollary 9.4.6]{S05}, \cite[Theorem 1.5]{DCRT19}.

The notion of \emph{unimodularity}~\cites{H97, BLPS99, BS01, AL07} of a random rooted graph can sometimes serve as a replacement for invariance properties. We do not provide the definition here (see~\cite{AL07}) but mention that Benjamini--Schramm limits (sometimes called \emph{sofic} graphs) are unimodular and it is a major open problem to determine if these are the only examples~\cite[Section 10]{AL07}. A theory of unimodular \emph{planar} maps is developed by Angel, Hutchcroft, Nachmias and Ray~\cites{AHNR16,AHNR18} and extended by Tim\'ar~\cite{T19} and Benjamini and Tim\'ar~\cite{BT19}. In~\cite[Corollary 3.2]{BT19}, ergodic unimodular plane triangulations, with finite expected degree of the root, are considered. Following Benjamini and Schramm~\cite{BS01-2}, it is shown that if such a map $G$ is nonamenable then $p_c(G)<1/2$ (and additional results regarding the number of infinite connected components).

\end{document}